\documentclass[twoside,a4paper,12pt,centertags]{amsart}
\usepackage{amsmath,amssymb,verbatim,vmargin}
\usepackage[bookmarks=true]{hyperref}   
\usepackage{esint}
\usepackage{tikz}

\theoremstyle{plain}
\newtheorem{thm}{Theorem}[section]
\newtheorem{lem}[thm]{Lemma}

\newtheorem{prop}[thm]{Proposition}

\newtheorem{lemma}[thm]{Lemma}

\theoremstyle{definition}
\newtheorem{defn}[thm]{Definition}
\newtheorem{rem}[thm]{Remark}

\newcommand{\osc}{\operatornamewithlimits{osc}}

\mathchardef\semic="303A
\newcommand{\R}{{\mathbf R}}

\newcommand{\Z}{{\mathbf Z}}

\newcommand{\mD}{{\mathcal D}}

\newcommand{\mL}{{\mathcal L}}

\newcommand{\V}{{\mathcal V}}

\newcommand{\sett}[2]{ \{ #1 \, \semic \, #2 \} }

\newcommand{\supp}{\text{{\rm supp}}\,}
\newcommand{\dist}{\text{{\rm dist}}\,}

\newcommand{\clos}[1]{\overline{#1}}

\newcommand{\barint}{\mbox{$ave \int$}}
\newcommand{\divv}{{\text{{\rm div}}}}

\newcommand{\esssup}{\text{{\rm ess sup}}}

\newcommand{\pd}{\partial}
\newcommand{\bdy}{\partial}

\newcommand{\bmo}{\text{BMO}}

\newcommand{\bvl}{{\text{BV}_\loc}}

\newcommand{\loc}{\text{{\rm loc}}}

\newcommand{\abs}[1]{| #1 |}
\newcommand{\ave}[1]{\langle #1 \rangle}

\def\barint_#1{\mathchoice
            {\mathop{\vrule width 6pt
height 3 pt depth -2.5pt
                    \kern -8.8pt
\intop}\nolimits_{#1}}%
            {\mathop{\vrule width 5pt height
3 pt depth -2.6pt
                    \kern -6.5pt
\intop}\nolimits_{#1}}%
            {\mathop{\vrule width 5pt height
3 pt depth -2.6pt
                    \kern -6pt
\intop}\nolimits_{#1}}%
            {\mathop{\vrule width 5pt height
3 pt depth -2.6pt
          \kern -6pt \intop}\nolimits_{#1}}}

\usepackage{color}

\definecolor{gr}{rgb}   {0.,   0.8,   0. }
\definecolor{bl}{rgb}   {0.,   0.5,   1. }
\definecolor{mg}{rgb}   {0.7,  0.,    0.7}

\begin{document}
\title[Bounded variation approximation of martingales and solutions]{Bounded variation approximation of $L_{\lowercase{p}}$ dyadic martingales and solutions to elliptic equations}
\author[Tuomas Hyt\"onen]{Tuomas Hyt\"onen}
\author[Andreas Ros\'en]{Andreas Ros\'en$\,^1$}
\thanks{$^1\,$Formerly Andreas Axelsson}
\address{Tuomas Hyt\"onen, Department of Mathematics and Statistics, P.O. Box 68 (Gustaf H\"all\-str\"omin katu 2b), FI-00014 University of Helsinki, Finland} 
\email{tuomas.hytonen@helsinki.fi}
\address{Andreas Ros\'en\\Mathematical Sciences, Chalmers University of Technology and University of Gothenburg\\
SE-412 96 G{\"o}teborg, Sweden}
\email{andreas.rosen@chalmers.se}

\begin{abstract}
We prove continuity and surjectivity of the trace map onto $L_p(\R^n)$, from a space
of functions of locally bounded variation, defined by the Carleson functional.
The extension map is constructed through a stopping time argument.
This extends earlier work by Varopoulos in the $\bmo$ case, related to the Corona theorem.
We also prove $L_p$ Carleson approximability results for solutions to elliptic non-smooth
divergence form equations, which generalize results in the case $p=\infty$ by Hofmann, Kenig,
Mayboroda and Pipher.
\end{abstract}



\thanks{T.H. is supported by the European Research Council through the ERC Starting Grant ``Analytic--probabilistic methods for borderline singular integrals''.
He is a member of the Finnish Centre of Excellence in Analysis and Dynamics Research.
A.R. is supported by Grant 621-2011-3744 from the Swedish research council, VR}
\maketitle

\section{Introduction}

Estimates of traces $u|_{\partial D}$ of functions $u:D\to \R$ in some given domain $D$, say in the Euclidean space, are important in analysis, for example in boundary value problems for partial differential equations. By local parametrization, it often suffices to consider the case where $D$ is the half-space
$$
  \R^{1+n}_+:= \sett{(t,x)}{t>0, x\in \R^n}
$$
and the traces are defined on the boundary $\partial\R^{1+n}_+=\R^n= \sett{(0,x)}{x\in \R^n}$. We shall concentrate on this case here.
A first problem is to show boundedness of the trace map
$$
  \gamma: u(t,x)\mapsto g(x) =(\gamma u)(x):= u(0,x).
$$
This amounts to identifying norms $\|\cdot\|_D$ and $\|\cdot\|_{\partial D}$ on the function spaces for $u$ and $g$ respectively, so that an estimate $\|g\|_{\partial D}\lesssim \|u\|_D$ holds.
A second problem is to determine whether $\gamma$, as a map
between the corresponding function spaces, is surjective. One wants that any $g$ can be extended to some $u$ in $D$ such that
$\gamma(u)=g$, with estimates
$\|u\|_D\lesssim \|g\|_{\partial D}$.

The most well-known trace result is the Sobolev trace theorem. This states that the trace map
$$
  \gamma: H^s(\R^{1+n}_+)\to H^{s-1/2}(\R^n)
$$
is bounded and surjective when $s>1/2$.
It is important to note that the Sobolev trace theorem breaks down in the limit case of regularity $s=1/2$,
and does not yield a bounded trace map onto the Lebesgue boundary space $L_2(\R^n)$.
One way to solve this problem is to consider instead the scale of Besov spaces $B_{p,q}^s$, where
the trace map
$$
  \gamma : B_{p,q}^s(\R^{1+n}_+)\to B_{p,q}^{s-1/p}(\R^n),
$$
is bounded and surjective when $s>1/p$.
Here also $\gamma : B_{p,1}^{1/p}(\R^{1+n}_+)\to L_p(\R^n)$ is bounded and surjective whenever $1\le p<\infty$, 
whereas the $L_2$ Sobolev scale of spaces is $H^s= B^s_{2,2}$.

Our first main result provides a new bounded and surjective trace map onto $L_p(\R^n)$, from a space of functions of locally bounded variation in the half-space, with norm
$$
  \|C(\nabla u)\|_{L_p(\R^n)},
$$
using the Carleson functional
\begin{equation*}
  C\mu(x):=\sup_{Q\owns x}\frac{1}{\abs{Q}}\iint_{\widehat{Q}} d\abs{\mu}(t,y)
\end{equation*}
for locally finite measures $\mu$ on $\R^{1+n}_+$.
Here the supremum is over all cubes $Q\subset\R^n$ containing $x$, and $\widehat Q:= (0,\ell(Q))\times Q$ denotes the Carleson box above $Q$, with
side length $\ell(Q)$.
Note that $C(\nabla u)$ is well defined for any $u\in \bvl(\R^{1+n}_+)$ of locally bounded variation.

In addition to the quantitative condition involving the norm $\|C(\nabla u)\|_{L_p(\R^n)}$, we also need some decay at infinity, which can be assumed in various forms. The weakest condition suitable for our needs is
\begin{equation*}
 \ave{\abs{u}}_{W(t,x)}:=  \frac{1}{\abs{W(t,x)}}\iint_{W(t,x)}\abs{u(s,y)}dsdy\underset{t\to\infty}{\longrightarrow} 0
\end{equation*}
for all $x\in\R^n$, where we use averages over Whitney regions
$$
  W(t,x):= \sett{(s,y)}{c_0^{-1}<s/t<c_0, |y-x|<c_1 t},
$$
with some fixed parameters $c_0>1$ and $c_1>0$. The above convergence is in particular implied by the stronger quantitative bound
\begin{equation*}
  Nu\in L_p(\R^n),
\end{equation*}
where
$$
  Nu(x):=\esssup_{|y-x|<t}|u(t,y)|,\quad x\in \R^n,
$$
 denotes the non-tangential maximal function. Indeed, it is easy to check that $\ave{\abs{u}}_{W(t,x)}\leq \inf_{y\in B(x,ct)}Nu(y)\to 0$ as $t\to\infty$ if $Nu\in L_p(\R^n)$. Thus we have the nested interior function spaces
\begin{equation*}
\begin{split}
  \mathcal{V}_p^0 &:=\{u\in\bvl(\R^{1+n}_+): C(\nabla u)\in L_p(\R^n),\ \ave{\abs{u}}_{W(t,x)}\underset{t\to\infty}{\longrightarrow} 0\text{ for all }x\in\R^n\} \\
\supset  \mathcal{V}_p^N &:=\{u\in\bvl(\R^{1+n}_+): C(\nabla u)\in L_p(\R^n),\ Nu\in L_p(\R^n)\} \\ 
\supset  \tilde{\mathcal{V}}_p^N &:=\{u\in C^1(\R^{1+n}_+)\,\quad: C(\nabla u)\in L_p(\R^n),\ Nu\in L_p(\R^n)\}.
\end{split}  
\end{equation*}
With the help of these spaces, we formulate our first $L_p$ extension result:

%
%
%
%

\begin{thm}   \label{thm:1}
Let $1<p<\infty$. Consider the normed linear function space $\V_p^0$ with norm $\|C(\nabla(\cdot))\|_p$.
Then the trace $\gamma u$ of any $u\in \V_p^0$, is well defined almost everywhere in the sense
of convergence of Whitney averages
$$
  (\gamma u)(x):= \lim_{t\to 0^+} |W(t,x)|^{-1}\iint_{W(t,x)} u(s,y) ds dy, \qquad x\in \R^n.
$$
The trace map $\gamma: \V_p^0\to L_p(\R^n)$ is well defined, and there exists $c_p<\infty$ so that estimates
$$
  \|\gamma u\|_{L_p(\R^n)}\le c_p \|C(\nabla u)\|_{L_p(\R^n)}
$$
hold for all $u\in\V_p^0$.
Moreover, the trace map $\gamma$ is surjective, and given any $g\in L_p(\R^n)$ there exists an extension
$u\in \V_p^0$ such that $\gamma u=g$, with estimates
$$
  \|C(\nabla u) \|_{L_p(\R^n)}\le c_p \|g\|_{L_p(\R^n)}.
$$
In fact, this extension may be chosen so that $u\in \tilde{\V}_p^N$, with the additional estimate
\begin{equation*}
  \|Nu \|_{L_p(\R^n)}\le c_p \|g\|_{L_p(\R^n)},
\end{equation*}
and pointwise non-tangential limits $\lim_{(t,y)\to(0,x), |y-x|<\alpha t} u(t,y)=g(x)$ exist at each Lebesgue point of $g$, for any fixed $\alpha<\infty$.
\end{thm}

We remark that the extension operator $g\mapsto u$ is non-linear, even though $\gamma$ itself of course is linear.


The corresponding trace result in the case $p=\infty$, proved by Varopoulos~\cite{Var1, Var2} is that there is a bounded and surjective trace map
$$
    \|u|_{\R^n}\|_{\bmo(\R^n)}\lesssim \|C(\nabla u) \|_{L_\infty(\R^n)},
$$
and a corresponding non-linear bounded extension operator,
where $\bmo(\R^n)$ stands for the John--Nirenberg space of functions of bounded mean
oscillation.
Following Varopoulos~\cite{Var2}, we obtain the extensions in Theorem~\ref{thm:1} from a result on approximate extensions of Lebesgue functions on $\R^n$.
This main component of Theorem~\ref{thm:1}, contained in our Theorem~\ref{thm:2}, generalizes well known techniques in the end point case $p=\infty$ related to the Corona Theorem, first proved by Carleson~\cite{Carleson:62}.
Our proof of Theorem~\ref{thm:2} though, is more in the spirit of Garnett~\cite[Ch. VIII, Thm. 6.1]{Gar}.

The statement below refers to dyadic versions of the non-tangential maximal functional, the Carleson functional and the Hardy--Littlewood maximal functional, defined respectively by
\begin{equation*}
\begin{split}
  N_{\mathcal D}f(x) &:=\sup_{Q:x\in Q\in\mathcal D}\sup_{(t,y)\in W_Q}\abs{f(t,y)},\qquad W_Q:=(\ell(Q)/2, \ell(Q))\times Q, \\
  C_{\mathcal D}f(x) &:=\sup_{Q:x\in Q\in\mathcal D}\frac{1}{\abs{Q}}\int_{\widehat Q}\abs{f(t,y)}dt\,dy,\qquad \widehat Q:=(0, \ell(Q))\times Q \\
  M_{\mathcal D}g(x) &:=\sup_{Q:x\in Q\in\mathcal D}\frac{1}{\abs{Q}}\int_{Q}\abs{g(y)}dy, \\
\end{split}
\end{equation*}
where $\mathcal D$ is a system of dyadic cubes in $\R^n$.

\begin{thm}  \label{thm:2}
  Fix $1<p<\infty$.
  Consider $g\in L_p(\R^n)$ and define the dyadic average extension
$$
  u(t,x):= \fint_Q g(y) dy, \qquad (t,x)\in W_Q,
$$
where $W_Q:= (\ell(Q)/2, \ell(Q))\times Q$ denotes the dyadic Whitney region above a dyadic cube $Q\subset \R^n$ of side length $\ell(Q)$.
Then, for any $0< \epsilon<1$, there exists  $f:\R^{1+n}_+\to\R$ which is constant on each dyadic Whitney region, with
pointwise estimates
\begin{equation*}
  \begin{cases}
    N_\mD(f-u)\le\epsilon M_\mD g, \\
    C_\mD(\nabla f)\lesssim \epsilon^{-1}M_\mD(M_\mD g),
\end{cases}
\end{equation*}
and implied norm estimates
$$
\begin{cases}
    \|N(f-u)\|_{L_p(\R^n)}\le \epsilon \|g\|_{L_p(\R^n)}, \\
    \|C(\nabla f)\|_{L_p(\R^n)}\lesssim \epsilon^{-1} \|g\|_{L_p(\R^n)}.
\end{cases}
$$
Moreover, for any fixed $\alpha<\infty$, non-tangential limits $\lim_{(t,y)\to(0,x), |y-x|<\alpha t}f(t,y)=:f(0,x)$ exist almost everywhere,
so that $\|f(0,\cdot)-g\|_{L_p(\R^n)}\le \epsilon \|g\|_{L_p(\R^n)}$.
\end{thm}


That the construction of approximate extensions $f$ as above with control of $C(\nabla f)$ is
indeed non-trivial, can be seen as follows.
Consider a ``lacunary'' function, which in a standard Haar basis would mean something like
$$
  u(x):= \sum_{Q\subset (0,1), |Q|\ge 2^{-k}} (\chi_{Q^l}(x)-\chi_{Q^r}(x)), \qquad x\in\R,
$$
where the sum is over dyadic subintervals of $(0,1)$ of length at least $2^{-k}$,
and the summand involves the characteristic functions of the left and right dyadic children of $Q$.
Then one checks that $\|u\|_p\lesssim \sqrt k$, whereas
the dyadic average extension $u(t,x)$ is seen to satisfy
$$
  C(\nabla u)(x)\gtrsim k, \qquad \text{for all } x\in(0,1).
$$
Therefore the dyadic average extension, or the closely related Poisson extension,
will not satisfy the required estimates.
Instead, Theorem~\ref{thm:2} is proved using a stopping time construction,
where we modify the stopping condition used in endpoint $\bmo$ case.

The second main result that we prove in this paper is an approximation result analogous to
Theorem~\ref{thm:2}, but with the dyadic martingale
$u$ replaced by a solution $u$ to an elliptic divergence form equation $\divv A\nabla u=0$ in $\R^{1+n}_+$.
This generalizes results in the end point case $p=\infty$ by Garnett~\cite{Gar} for the Laplace equation in $\R^2_+$, by Dahlberg~\cite{D3} for the Laplace equation on Lipschitz domains in $\R^n$, and by Kenig, Koch, Pipher and Toro~\cite{KKPT} and Hofmann, Kenig, Mayboroda and Pipher~\cite{HKMP} for divergence form equations on Lipschitz domains.

\begin{thm}  \label{thm:3}
  Fix $1<p<\infty$ and coefficients $A\in L_\infty(\R^n;\mL(\R^{1+n}))$ which are accretive in the sense
  that there exists $\lambda_A>0$ such that
  $$
     (A(x)v,v)\ge \lambda_A|v|^2,
  $$
  for almost all $x\in\R^n$ and all $0\ne v\in \R^{1+n}$.
Then for any $0< \epsilon<1$, there exists $c_\epsilon<\infty$ such that the following holds.
Given any weak solution $u:\R^{1+n}_+\to \R$ to the $t$-independent real scalar, but possibly non-symmetric, divergence form elliptic equation
$$
  \divv_{t,x} A(x)\nabla_{t,x} u(t,x)=0,
$$
with estimates $Nu\in L_p(\R^n)$,
there exists a function $f$ in $\R^{1+n}_+$ of locally bounded variation, with estimates
$$
\begin{cases}
    \|N(f-u)\|_{L_p(\R^n)}\le \epsilon \|Nu\|_{L_p(\R^n)}, \\
    \|C(\nabla f)\|_{L_p(\R^n)}\leq c_{\epsilon} \|Nu\|_{L_p(\R^n)}.
\end{cases}
$$
\end{thm}

We shall informally refer to such an $f$ as an \emph{approximant} of $u$. Note that the functions $f$ and $u$ share the same domain of definition, $\R^{1+n}_+$, and no extensions are involved here, in contrast to Thm. \ref{thm:2}, which established an \emph{approximate extension} $f$ on $\R^{1+n}_+$ of an initial function $g$ on $\R^n$. While Thm. \ref{thm:2} also featured a function $u$ on $\R^{1+n}_+$ in a seemingly similar role as in Thm. \ref{thm:3}, the actual position of $u$ in Thm. \ref{thm:2} was mainly auxiliary, as an intermediate object in the construction of $f$, while in Thm. \ref{thm:3} we regard the solution $u$ itself, rather than its boundary limit, as the primary object of interest that we wish to approximate.

Again, our proof in fact gives the pointwise bounds
\begin{equation*}
  N_{\mD^\delta}(f-u)\le \epsilon M_{\mD^\delta}(Nu)\quad \text{and}\quad
  C_{\mD^\delta}(\nabla f)\leq c_\epsilon M_{\mD^\delta}(Nu),
\end{equation*}
where $\mD^\delta$ denotes modified dyadic versions of maximal and Carleson functionals, see Section~\ref{sec:solutions},
from which the asserted norm bounds are immediate by the maximal inequality.
The dependence of $c_\epsilon$ on $\epsilon$ given by our proof, is certainly worse than $1/\epsilon$.
We also note that, when $p$ is large enough, it follows from the work \cite{HKMP} that the $L_p$ Dirichlet problem is well posed since the $L$-harmonic measure, for $L= -\divv_{t,x} A(x)\nabla_{t,x}$, is in $A_\infty(dx)$, and therefore the norm $\|Nu\|_{L_p(\R^n)}$ can be replaced by $\|u|_{\R^n}\|_{L_p(\R^n)}$ in Theorem~\ref{thm:3} in this case.

The previously studied end point case $p=\infty$ of Theorem~\ref{thm:3} is a key tool in the study of boundary value problems for elliptic equations $Lu=0$ as above in  \cite{KKPT, HKMP}.
The existence of such an approximant $f\approx u$ in the $L_\infty$ norm is referred to as \emph{$\epsilon$-approximability} (of solutions), and it is proved in
 \cite{KKPT} that it implies the $A_\infty$ property of the harmonic measure.
However, the latter is well known by Dahlberg, Jerison and Kenig~\cite{DJK} to imply the comparability of non-tangential
maximal functions  and square functions of solutions $u$, which in turn is a key tool in  \cite{KKPT, HKMP} in the construction
of approximants $f\approx u$.
Therefore bounded variation approximability $f\approx u$, $A_\infty$ control of harmonic measure and $N\approx S$ comparability turn out to be equivalent. The importance of the approximability property, through this circle of arguments in the $p=\infty$ case, motivates our generalization in Theorem~\ref{thm:3} to the case $p<\infty$.
We show that these three properties are also equivalent 
when $p=\infty$ in the approximability property is replaced by $n/(n-1) \leq p<\infty$.
See Section~\ref{sec:harmmeas} for the detailed statements. In particular, we extend \cite[Thm. 2.3]{KKPT} to
$n/(n-1)\leq p<\infty$.

The outline of the paper is as follows. In Section~\ref{sec:NAC}, we survey the basic
estimates for the functionals defining our spaces. In Section~\ref{sec:ext} we deduce
Theorem~\ref{thm:1} from Theorem~\ref{thm:2}, and prove the latter using
a weighted stopped square function estimate,
the proof of which is in Section~\ref{sec:stoppedsqfcnmart}.
Finally in Section~\ref{sec:solutions} we prove Theorem~\ref{thm:3}, using local $N\approx S$
estimates which we borrow from \cite{HKMP}, as a replacement of the stopped square function estimates which we used in the martingale case of Theorem~\ref{thm:2}.

\section{The basic functionals}   \label{sec:NAC}

In this section, we collect well known facts concerning the functionals that we use to define several norms of functions
in the half space $\R^{1+n}_+$.

First we fix notation.
We write the $L_p(\R^n)$ norm as $\|\cdot\|_p$.
Cubes in $\R^n$ (dyadic or not) we denote by $Q, R, S, \ldots$, and we assume that these are open.
The Carleson box above a cube $Q\subset \R^n$ is denoted
$$
  \widehat Q:= (0,\ell(Q))\times Q\subset \R^{1+n}_+,
$$
where $\ell(Q)$ denotes the sidelength of $Q$.
We write $cQ$ to denote the cube with same center as $Q$ but with $\ell(cQ)= c\ell(Q)$.

Let $\mD= \bigcup_{j\in \Z}\mD_j$ denote a system of dyadic cubes in $\R^n$, with
$\mD_j$ being the cubes of sidelength $\ell(Q)=2^{-j}$, such that the dyadic cubes in $\mD$ form a connected tree under inclusion.
Let $W_Q:= (\ell(Q)/2, \ell(Q))\times Q$ denote dyadic Whitney regions.
The corresponding non-dyadic Whitney region around a point $(t,x)\in\R^{1+n}_+$ we define to be
$$
W(t,x):= \sett{(s,y)}{c_0^{-1}<s/t<c_0, |y-x|<c_1t},
$$
where $c_0>1$ and $c_1>0$ are fixed parameters.

The Hardy-Littlewood maximal function of $f\in L_1^\loc(\R^n)$ that we use is
$$
 Mf(x):= \sup_{Q\ni x} \fint_Q |f(y)| dy,
$$
where the supremum in $Mf$ is over all cubes $Q\subset \R^n$ containing $x$.
Restricting the cubes to the dyadic ones in the supremum yields the dyadic Hardy-Littlewood maximal function
$M_\mD f(x)$.
We also require the following truncated (to large cubes) version of the dyadic Hardy-Littlewood maximal function:

\begin{equation}   \label{eq:truncmax}
  M_\mD f(Q)=  \sup_{R\supset Q, R\in \mD} \fint_R |f(y)| dy,\qquad Q\in\mD.
\end{equation}

\begin{defn}   \label{defn:contfuncNCA}
For a locally integrable function $f(t,x)$ in $\R^{1+n}_+$ we define, for $x\in\R^n$,
the {\em non-tangential maximal functional} $Nf$,  the {\em Carleson functional} $Cf$ and
the {\em area functional} $Af$ as
\begin{gather*}
  Nf(x):= \esssup_{|y-x|<\alpha t} |f(t,y)|, \\
  Cf(x):= \sup_{Q\ni x} |Q|^{-1} \iint_{\widehat Q} |f(t,y)| dtdy, \\
  Af(x):= \iint_{|y-x|<\alpha t} |f(t,y)| t^{-n}dtdy.
\end{gather*}
Here the supremum in $Cf$ is over all cubes $Q\subset \R^n$ containing $x$.
 In the definition of $Nf$ and $Af$, the parameter $\alpha>0$ denotes some fixed aperture of the cones. To emphasize the exact dependence of the aperture, we sometimes write $N^{(\alpha)}$ and $A^{(\alpha)}$.

For a function $f(t,x)$ in $\R^{1+n}_+$, having constant value $f_Q$ on each dyadic Whitney region $W_Q$, we also define dyadic versions of these functionals by
\begin{gather*}
  N_\mD f(x):= \sup_{x\in Q\in\mD} |f_Q|, \\
  C_\mD f(x):= \sup_{x\in Q\in \mD} |Q|^{-1} \sum_{R\in\mD, R\subset Q} |f_R| |W_R|, \\
  A_\mD f(x):= \sum_{x\in Q\in\mD} |f_Q| \ell(Q).
\end{gather*}
\end{defn}

The reader is invited to check that these definition agree with those given in the Introduction without assuming the constant values over the regions $W_Q$.

We want to point out that, throughout this paper, we are using the measure $dtdx$ and not the measure $t^{-1}dtdx$, although the latter is quite common in the literature.
Note that the functionals $A$ and $C$ extend in a natural way to the case when $f$ is a signed measure on $\R^{1+n}_+$, and in particular to the case of gradients of functions of locally bounded variation.

We record the following norm equivalences between different choices for the aperture of the cones.

\begin{prop}  \label{prop:aperture}
Fix $0<\alpha, \beta<\infty$.
Then for any $1\le p\le \infty$, we have
$
\|N^{(\alpha)}f\|_p \approx \|N^{(\beta)}f\|_p,
$
and for any $1\le p< \infty$, we have
$
  \|A^{(\alpha)}f\|_p \approx  \|A^{(\beta)}f\|_p.
$
\end{prop}

\begin{proof}
The estimates for $Nf$ are proved in Fefferman and Stein~\cite[Lem. 1]{FS}.
To prove the estimate for the $A$-functional, we follow
Coifman, Meyer and Stein~\cite[Prop. 4, case $2\le p<\infty$]{CMS} and consider $0<\alpha<\beta<\infty$: Dualize against $\|h\|_{p'}=1$ to get
\begin{multline*}
  \|A^{(\beta)} f\|_p= \int_{\R^n}\left( \iint_{|y-x|<\beta t}|f(t,y)| t^{-n}dtdy \right) h(x) dx\\
  = \iint_{\R^{1+n}_+} |f(t,y)| \left( t^{-n}\int_{|x-y|<\beta t}h(x) dx  \right) dtdy \\
  \lesssim  \iint_{\R^{1+n}_+} |f(t,y)| \left(t^{-n} \int_{|x-y|<\alpha t} Mh(x) dx  \right) dtdy \\
  = \int_{\R^n} A^{(\alpha)} f(x) Mh(x) dx\lesssim \|A^{(\alpha)} f\|_p.
\end{multline*}

\end{proof}

We also record the following equivalence of norms between the corresponding dyadic and non-dyadic functionals.
\begin{prop}   \label{prop:dyadicvsnondya}
 We have
\begin{align*}
  \|N f\|_p\approx \|N_\mD f\|_p, & \qquad 1\le p\le \infty, \\
  \|C f\|_p\approx \|C_\mD f\|_p, & \qquad 1< p\le \infty, \\
  \|A f\|_p\approx \|A_\mD f\|_p, & \qquad 1\le p< \infty,
\end{align*}
uniformly for all functions $f(t,x)$ in $\R^{1+n}_+$ that are constant on each dyadic Whitney region.
\end{prop}

\begin{proof}
  For proofs of the results for $N$ and $C$, we refer to \cite{HR}.
  Consider now the area functional $A$. As in the proof of Proposition~\ref{prop:aperture}, the proof is an adaption of \cite[Prop. 4, case $2\le p<\infty$]{CMS}.
Dualize against $\|h\|_{p'}=1$ to get
\begin{equation*}
\begin{split}
  \|A f\|_p &= \int_{\R^n}\left( \iint_{|y-x|<\alpha t}|f(t,y)| t^{-n}dtdy \right) h(x) dx\\
  &= \iint_{\R^{1+n}_+} |f(t,y)| \left(t^{-n} \int_{|x-y|<\alpha t}h(x) dx  \right) dtdy \\
 &\lesssim \sum_{Q\in\mD} |f_Q| |W_Q|\left(  \fint_Q (Mh)(x) dx\right)  \\
  &= \int_{\R^n} (A_\mD f) (Mh) dx\lesssim \|A_\mD f\|_p.\qedhere
\end{split}
\end{equation*}
\end{proof}

Less obvious is the following important $L_p$ equivalence of the $A$ and $C$ functionals.

\begin{prop}     \label{prop:AGequiv}
For $1\le p<\infty$, we have
$$
\|Af\|_p\lesssim \|Cf\|_p.
$$
For $1<p\le \infty$, we have
$$
\|Cf\|_p\lesssim \|Af\|_p
$$
(for any fixed aperture $\alpha$ in the case $p=\infty$).
\end{prop}

At the endpoint $p=\infty$, the $A$-functional depends on the choice of aperture,
and should be replaced by the Carleson functional,
which is strictly smaller, as seen from the example
$$
  f(t,x)= (t+|x|)^{-n}.
$$
At the endpoint $p=1$, we have $\|Cf\|_1<\infty$ only if $f=0$, so the Carleson functional should be replaced by the area functional, which in this case defines simply the function space
$L_1(\R^{1+n}_+)$.

Proposition~\ref{prop:AGequiv} is a reformulation of Coifman, Meyer and Stein~\cite[Thm. 3]{CMS}.
The proof below contains some novelties in the estimate $A\lesssim C$, by using a duality argument rather than a good lambda estimate.

\begin{proof}[Proof of Proposition~\ref{prop:AGequiv}]
  For $C\lesssim A$ we have
\begin{multline*}
  M(Af)(x)=\sup_{Q\ni x} \fint_Q \left( \iint_{|y-x|<\alpha t} |f(t,y)| t^{-n}dtdy\right) dx \\
  = \sup_{Q\ni x} |Q|^{-1}\iint \left( t^{-n}\int_{|x-y|<\alpha t, x\in Q} dx \right) |f(t,y)| dtdy\gtrsim Cf(x).
\end{multline*}

For $A\lesssim C$, we argue by duality with a suitable $\|h\|_{p'}=1$:
\begin{equation*}
\begin{split}
  \|Af\|_{p}
  &=\int_{\R^n}\Big(\iint_{|y-x|<\alpha t}|f(t,y)|t^{-n}dt dy\Big)h(x)dx \\
  &=\iint_{\R_+^{1+n}}|f(t,y)|\Big(t^{-n}\int_{|x-y|<\alpha t}h(x)dx\Big)dt dy \\
  &=:\iint_{\R_+^{1+n}}|f(t,y)| H(t,y)dt dy
    =\int_0^\infty\Big(\iint_{\substack{(t,y)\in \R_+^{1+n} \\ H(t,y)>\lambda}} |f(t,y)| dt dy\Big) d\lambda,
\end{split}
\end{equation*}
where $H(t,y)$ is defined by the penultimate equality to be the parenthetical quantity on the line above.
If $H(t,y)>\lambda$, there is a cube $Q$ such that $\fint_Q h(x)dx>c\lambda$ and $(t,y)\in\widehat Q$. By the Whitney covering lemma, there is a collection $\mathcal{Q}_\lambda$ of these cubes such that the $\widehat Q$ are pairwise disjoint, and the $5\widehat Q$, $Q\in\mathcal{Q}_\lambda$, cover all the points $(t,y)$ with $H(t,y)>\lambda$. Thus
\begin{equation*}
\begin{split}
   \iint_{\substack{(t,y)\in \R_+^{1+n} \\ H(t,y)>\lambda}} |f(t,y)| dt dy
   &\leq\sum_{Q\in\mathcal{Q}_\lambda}\iint_{5\widehat Q}|f(t,y)|dt dy
   \leq\sum_{Q\in\mathcal{Q}_\lambda}|5 Q|\inf_{x\in Q} Cf(x) \\
   &\lesssim \sum_{Q\in\mathcal{Q}_\lambda}\int_{Q} Cf(x) dx
     \leq\int_{\{Mh>c\lambda\}}Cf(x) dx.
\end{split}
\end{equation*}
Substituting back, this shows that
\begin{equation*}
\begin{split}
  \|Af\|_{p}
  \lesssim\int_0^\infty \int_{\{Mh>c\lambda\}}Cf(x) dx d\lambda
  \lesssim\int_{\R^n} Cf(x)Mh(x)dx
  \lesssim \|Cf\|_{p}.\qedhere
\end{split}
\end{equation*}
\end{proof}

\section{A dyadic weighted stopped square function estimate}   \label{sec:stoppedsqfcnmart}

In this section we prove an auxiliary weighted norm inequality, which will be used in the construction of the extensions in the subsequent section.

Let $\omega_*\subset \mD$ be any collection of dyadic cubes.
Given any $Q\in\mD$, define its stopping parent $Q_*$ to be the minimal $Q_*\in \omega_*$ such that $Q_*\supsetneqq Q$.
If no such $Q_*$ exists, we let $Q_*:= Q$.
Define the stopped square function
$$
  S_{\omega_*} u(x):= \left( \sum_{Q\in \omega_*} |u_Q-u_{Q_*}|^2 1_Q(x)  \right)^{1/2}, \qquad x\in\R^n.
$$

\begin{lem}   \label{lem:nonweightstopSF}
  The stopped square function $S_{\omega_*}$ defined above, has estimates
  \begin{align*}
    | \{S_{\omega_*}u>\lambda\}| & \lesssim \lambda^{-1} \|u\|_{L_1(\R^n)},\qquad \lambda>0, \\
    \| S_{\omega_*}u\|_{L_2(\R^n)} & \lesssim \|u\|_{L_2(\R^n)},
  \end{align*}
  uniformly for any collection of dyadic cubes  $\omega_*$.
\end{lem}

   A standard Calder\'on--Zygmund decomposition argument yields the weak $L_1$ estimate, given the $L_2$ estimate.
   This $L_2$ estimate is in turn proved by a well known martingale square functions estimate,
see for example Garnett~\cite[Ch. VIII, Lem. 6.4]{Gar}.
For completeness, we include the details of the proof.

\begin{proof}
(a)
   For the $L_2$ estimate, we write $\omega_*=\bigcup_{k=-\infty}^\infty \omega_k$,
   where the cubes in $\omega_k$ are disjoint and $\omega_{k-1}= \sett{Q_*}{Q\in\omega_k}$.
   We define the martingale $\{u_k\}_{k=-\infty}^\infty$, where
   $$
     u_k(x):=
     \begin{cases}
       \fint_Q u(y) dy, & \qquad x\in Q\in \omega_k, \\
       u(x), & \qquad x\notin\bigcup_{Q\in \omega_k}Q.
     \end{cases}
   $$
   This yields
   \begin{multline*}
      \|S_{\omega_*}u\|_2^2=\sum_k \sum_{Q\in \omega_k} |u_Q-u_{Q_*}|^2 |R|
      \le \sum_k \int_{\R^n}|u_{k+1}-u_k|^2 dx \\
      =\sum_k\int_{\R^n} (u_{k+1}^2+u_k^2-2u_{k+1}u_k)dx
      =\sum_k\int_{\R^n} (u_{k+1}^2-u_k^2)dx\le \int_{\R^n} u^2 dx,
   \end{multline*}
   where we have used that $\int u_ku_{k+1} dx= \int u_k^2 dx$.

(b)
Let $Q_k$ denote the maximal dyadic cubes contained in $\{M_\mD u>\lambda\}$.
Write $u=g+\sum_k b_k$, where $|g|\le \lambda$ and $\supp b_k\subset Q_k$ with $\int_{Q_k}b_k=0$.
The stated estimate follows from the two estimates
$$
  |\{S_{\omega_*}g>\lambda/2\}|\lesssim \lambda^{-2}\int |S_{\omega_*}g|^2 dx\lesssim
  \lambda^{-2}\int |g|^2 dx\lesssim \lambda^{-1}\int |g| dx\le \lambda^{-1}\int |u| dx,
$$
using (a) and that $\int_{Q_k}|u|dx\approx\lambda|Q_k|$, and
$$
  |\{S_{\omega_*}(\sum_k b_k)>\lambda/2\}|\lesssim\sum_k|Q_k|=|\{M_\mD u>\lambda\}|\lesssim  \lambda^{-1}\int |u| dx,
$$
using that $\supp S_{\omega_*} b_k\subset Q_k$ and the weak $L_1$ bound of the Hardy--Littlewood maximal function.
\end{proof}

The main result of this section is the following weighted estimate for $S_{\omega_*}$, inspired by
the work of Gundy and Wheeden~\cite[Thm. 2]{GW} for the non-stopped square function (i.e., case $\omega_*=\mathcal{D}$).

\begin{prop}  \label{prop:stoppedweightedSF}
  Fix a Muckenhoupt weight $w\in A_\infty(dx)$ and an exponent  $1\le p<\infty$.
  Then we have
  the stopped square function estimate
  $$
    \| S_{\omega_*} u\|_{L_p(\R^n;w)}\lesssim \|M_\mD u\|_{L_p(\R^n;w)},
  $$
  uniformly for any collection of dyadic cubes  $\omega_*$.
\end{prop}

\begin{proof}
   It suffices to prove a good lambda inequality
   $$
     w(\{ S_{\omega_*}u>2\lambda, M_\mD u<\gamma\lambda \})\lesssim \gamma^\delta w(\{ S_{\omega_*}u>\lambda \}),
   $$
   for some $\delta>0$. By the $A_\infty$ assumption, this will follow from
   a Lebesgue measure estimate
   $$
     | \{ S_{\omega_*}u>2\lambda, M_\mD u<\gamma\lambda \} \cap Q |\lesssim \gamma |Q|,
   $$
   for any maximal dyadic cube $Q\subset \{ S_{\omega_*}u>\lambda \}$.
   To this end, assume that $x\in\{ S_{\omega_*}u>2\lambda, M_\mD u<\gamma\lambda \} \cap Q$.
   Then
\begin{multline*}
  4\lambda^2 <\sum_{R\in \omega_*, R_*\subset Q} |u_R-u_{R_*}|^2 1_R(x)
  + \sum_{R\in \omega_*, R\subset Q\subsetneqq R_*} |u_R-u_{R_*}|^2 1_R(x)  \\
  + \sum_{R\in \omega_*, R\supsetneqq Q} |u_R-u_{R_*}|^2 1_R(x)\le
  S_{\omega_*}(u1_Q)(x)+ 4(\gamma\lambda)^2+ \lambda^2,
\end{multline*}
using that $M_\mD u(x)<\gamma\lambda$ for the second term and the maximality of $Q$ for the last term.
Therefore, assuming $\gamma<1/2$, we have $S_{\omega_*}(u1_Q)(x)> \lambda$, so
$$
   \{ S_{\omega_*}u>2\lambda, M_\mD u<\gamma\lambda \} \cap Q\subset \{S_{\omega_*}(u1_Q)>\lambda\}.
$$
From Lemma~\ref{lem:nonweightstopSF}, we get the estimate
$$
  | \{S_{\omega_*}(u1_Q)>\lambda\} |\lesssim \lambda^{-1} \int_Q |u| dx.
$$
We may assume that $\{ S_{\omega_*}u>2\lambda, M_\mD u<\gamma\lambda \} \cap Q\ne \emptyset$, and in particular
that $ \int_Q |u| dx\le \gamma\lambda |Q|$.
Put together, this proves that $|\{ S_{\omega_*}u>2\lambda, M_\mD u<\gamma\lambda \} \cap Q|\lesssim \gamma |Q|$.
\end{proof}

\section{Construction of extensions}   \label{sec:ext}

In this section we prove Theorems~\ref{thm:1} and \ref{thm:2}, assuming a square function estimate which
we prove in Section~\ref{sec:stoppedsqfcnmart}.
We first prove Theorems~\ref{thm:1}, where we use Theorem~\ref{thm:2} in the construction of extensions.

\begin{proof}[Proof of Theorem~\ref{thm:1}, part I: existence and bound of the trace of $u\in \V_p^0$]
  We fix $x\in \R^n$ and consider two Whitney regions $W(t_1,x)$ and $W(t_2,x)$, with $t_1<t_2$.
  Estimate
\begin{multline}  \label{eq:whitneydifference}
  \left||W(t_2,x)|^{-1} \iint_{W(t_2,x)} u(s,y) dsdy- |W(t_1,x)|^{-1} \iint_{W(t_1,x)} u(s,y) dsdy\right| \\
  \approx \left| \iint_{W(1,0)} (u(t_2s, x+t_2y)-u(t_1s, x+t_1 y)) dsdy\right| \\
   \lesssim \iint_{W(1,0)}\int_{t_1}^{t_2} |\nabla u(ts, x+ty)| dt dsdy \\
 = \int_{t_1}^{t_2} \int_{t/c_0}^{c_0t}  \int_{|y'-x|<c_1 t} |\nabla u(s', y')| t^{-1-n} dy'ds'dt \\
 \lesssim  \int_{t_1/c_0}^{c_0t_2}  \int_{|y'-x|<c_1 s'} |\nabla u(s', y')| (s')^{-n} dy'ds'
\end{multline}

  Since $A(\nabla u)\in L_p(\R^n)$ by Proposition~\ref{prop:AGequiv}, we have for almost all $x\in\R^n$ that
  $A(\nabla u)(x)<\infty$. For such $x$, it follows from the above estimate that Whitney averages
  converges as $t\to 0$. Thus, in this sense we have a well defined trace almost everywhere on $\R^n$.
  
 On the other hand, directly from the definition of the space $\V_p^0$ it follows
 that $|W(t_2,x)|^{-1} \iint_{W(t_2,x)} u(s,y) dsdy\to 0$ as $t_2\to \infty$.
  The estimate
$$
  \|\gamma u\|_p\lesssim \|A(\nabla u)\|_p\approx \|C(\nabla u)\|_p
$$
  follows.
Indeed, if the right hand side is finite, then at almost every $x\in\R^n$, we see from \eqref{eq:whitneydifference}
that the trace $\gamma u(x)$ exists in the sense of convergence of Whitney averages, since the right hand side 
in \eqref{eq:whitneydifference} has limit zero as $t_2\to 0$.
Then letting $t_2\to\infty$ shows the pointwise estimate $|\gamma u(x)|\lesssim |A(\nabla u)(x)|$. 

This completes the first part of the proof of Theorem \ref{thm:1}.
\end{proof}

\begin{proof}[Proof of Theorem~\ref{thm:1}, part II: construction of the extension assuming Theorem \ref{thm:2}]

  We construct the extension $u$ of $g\in L_p(\R^n)$ as follows.
  Define functions $g_k$, $u_k$ and $f_k$, $k=0,1,2,\ldots$ inductively:
  Let $g_0:= g$. Given $g_k\in L_p(\R^n)$, $k\ge 0$, we apply Theorem~\ref{thm:2} to define
  the dyadic extension $u_k$ and its approximation $f_k$, with estimates
\begin{gather*}
  \|N(f_k-u_k)\|_p\le \epsilon \|g_k\|_p, \\
  \|C(\nabla f_k)\|_p\le C\epsilon^{-1}\|g_k\|_p.
\end{gather*}
Then let $g_{k+1}:= g_k-f_k|_{\R^n}$.
We have
$$
  \|g_{k+1}\|_p\le \|N(u_k-f_k)\|_p\le \epsilon \|g_k\|_p
$$
and therefore $\|g_k\|_p\le \epsilon^k \|g\|_p$.
Define
$$
  f:= \sum_{k=0}^\infty f_k.
$$
This is an exact extension of $g$ since
$0 = \lim_{k\to \infty}\|g_{k+1}\|_p= \Big\| g-\sum_{j=0}^k f_j|_{\R^n} \Big\|_p$.
Moreover, we have the estimate
$$
  \|C(\nabla f)\|_p
  \leq \sum_{k=0}^\infty\|C(\nabla f_k)\|_p
  \leq\sum_{k=0}^\infty C\epsilon^{-1}\|g_k\|_p 
  \leq \sum_{k=0}^\infty C\epsilon^{-1} \epsilon^k \|g\|_p\lesssim \|g\|_p,
$$
fixing some $0<\epsilon <1$, and  similarly
\begin{equation*}
  \|N f\|_p\leq\sum_{k=0}^\infty\|Nf_k\|_p
  \leq\sum_{k=0}^\infty\big(\|Nu_k\|_p+\epsilon\|g_k\|_p\big)
  \lesssim\sum_{k=0}^\infty\|g_k\|_p\lesssim\|g\|_p.
\end{equation*}
This shows that we have an extension $f\in\V_p^N$.

It remains to mollify $f$ to obtain another extension
$$
   u(t,x):= \iint_{\R^{1+n}_+} f(ts,x+ty) \eta(s,y) dsdy, \qquad (t,x)\in\R^{1+n}_+,
$$
where $\eta\in C_0^\infty(W(1,0))$ has $\iint \eta =1$.
Then it is straightforward to verify that $u\in \tilde{\V}_p^N$, with the stated estimate
of
\begin{equation*}
  \nabla u(t,x)= \iint_{\R^{1+n}_+} \begin{bmatrix} s & y^t\\ 0 & I \\\end{bmatrix}\nabla f(ts, x+ty) \eta(s,y)dsdy.\qedhere
\end{equation*}
\end{proof}

For the proof of Theorem~\ref{thm:2}, we require the following lemma
for the truncated dyadic maximal function from \eqref{eq:truncmax}.

\begin{lem}   \label{lem:truncmax}
  For any $g\in L_1^\loc(\R^n)$ and $Q\in \mD$, we have
$$
  \frac{|Q|}{M_\mD g(Q)}\le 4\int_{Q} \frac {dx}{M_\mD g(x)}.
$$
\end{lem}

\begin{proof}
Define
$$
  E_Q:= \sett{x\in Q}{M_\mD g(x)>2 M_\mD g(Q)}= \sett{x\in Q}{M_\mD (g1_Q)(x)>2 M_\mD g(Q)}.
$$
The weak $L_1$ estimate for $M_\mD$ yields
$$
  |E_Q|\le \frac 1{2M_\mD g(Q)}\int_Q |g| dx\le \frac 12|Q|.
$$
Thus
\begin{equation*}
  \frac{|Q|}{M_\mD g(Q)}\le 2\frac{|Q\setminus E_Q|}{M_\mD g(Q)}\le 4\int_{Q\setminus E_Q} \frac {dx}{M_\mD g(x)}\le 4\int_{Q} \frac {dx}{M_\mD g(x)}.\qedhere
\end{equation*}
\end{proof}

The following lemma shows that the horizontal derivatives are essentially controlled by the vertical ones, reducing the crux of the matter to controlling the latter.

\begin{lem}  \label{lem:tanggradest}
Let $u$ be a function in $\R^{1+n}_+$ which is constant on dyadic Whitney regions,
and let $Q\in \mD$.
Uniformly for such $u$ and $Q$, we have the estimate
$$
  \iint_{\clos{\widehat Q}}|\nabla_x u| dtdx \lesssim  \iint_{\clos {\widehat Q}}|\pd_t u| dtdx+
  |Q|\sum_{Q'} |u_{Q'}|,
$$
where the last sum is over $Q'\in\mD$ with $\ell(Q')=\ell(Q)$, $\partial Q'\cap \partial Q\ne\emptyset$.
\end{lem}

Note the obvious meaning of $\iint_{\clos{\widehat Q}}$: The contribution from $\bdy \widehat Q\cap\R^{1+n}_+$ is to be counted.

\begin{proof}
Fix a dyadic cube $Q$.
Consider a contribution to $\nabla_x u$ from the jump across $\partial W_R\cap \partial W_S\subset \clos{\widehat Q}$, where $\ell(R)=\ell(S)$.
Go up through ancestors to a common dyadic ancestor $R_N=S_N$, and write
$$
  R= R_0\subset R_1\subset \ldots \subset R_N= S_N\supset \ldots \supset S_1\supset S_0=S.
$$
If $R_N=S_N\subset Q$,
then we estimate
$$
  |u_R-u_S| |R|\lesssim \sum_{k=1}^{N}2^{-nk}( |u_{R_{k}}-u_{R_{k-1}}| |R_k|)+
  \sum_{k=1}^{N}2^{-nk}( |u_{S_{k}}-u_{S_{k-1}}| |S_k|).
$$
For some fixed sub-cube $R_k\subsetneqq Q$, there arise in this way one such term $|u_{R_{k}}-u_{R_{k-1}}| |R_k|$
from each sub-cube $R$ of $R_k$ such that $(\partial R)\cap (\partial R_k)\neq \emptyset$.
There are at most $C 2^{(n-1)k}$ such sub-cubes with $\ell(R)= 2^{-k}\ell(R_k)$.

If $R_N=S_N\not\subset Q$, then we estimate as above, but stop at $|R_K|= |S_K|= |Q|$,
and we obtain two extra terms
$$
  2^{-nK} |u_{R_K}| |Q| + 2^{-nK} |u_{S_K}| |Q|.
$$
Summing up, using $\sum_0^\infty2^{(n-1)k}2^{-nk} =2$, we get
$$
  \iint_{\clos{\widehat Q}}|\nabla_x u| dtdx \lesssim  \iint_{\clos {\widehat Q}}|\pd_t u| dtdx+
  |Q|\sum_{Q'} |u_{Q'}|.
$$
\end{proof}

Now we are fully prepared for:

\begin{proof}[Proof of Theorem~\ref{thm:2}]

(1)
We first localize the problem to a large top cube $Q_0$.
Choose $Q_0\in\mD$ large enough so that
$$
  \int_{\R^n\setminus Q_0} |M_\mD g|^p dx \le \delta,
$$
where $\delta>0$ is to be chosen below.
Define
$$
  g_2(x):=
  \begin{cases}
    \fint_{Q_0} g dy, & \qquad x\in Q_0,\\
    g(x), & \qquad x\notin Q_0,
  \end{cases}
$$
and let $g_1:= g-g_2$.
Let $u_1$ and $u_2$ be the dyadic extensions of $g_1$ and $g_2$ respective, so that $u_1$ is non-zero only on $\widehat Q_0$.

Define the approximation $f_2$ of $u_2$ to be
$$
  f_2(t,x):=
  \begin{cases}
    \fint_{Q_0} g dy, & \qquad (t,x)\in \widehat Q_0, \\
    0, & \qquad  (t,x)\notin \widehat Q_0.
  \end{cases}
$$
Then $N_\mD(f_2-u_2)(x)\le M_\mD g(x)$ if $x\notin Q_0$ and $N_\mD(f_2-u_2)(x)\le \inf_{Q_1} M_\mD g$ if $x\in Q_0$,
where $Q_1$ is the sibling of $Q_0$.
Thus
$$
\|N_\mD(f_2-u_2)\|_p^p\le 2 \int_{\R^n\setminus Q_0} |M_\mD g|^p dx \le  (\epsilon/2)^p \|g\|_p^p,
$$
provided $2\delta\le (\epsilon/2)^p \int_{\R^n} |g|^p dx$.
Furthermore
$\|C(\nabla f_2)\|_p\lesssim |Q_0|^{1/p}(|Q_0|^{-1}\int_{Q_0} g dy)\le \|g\|_p$.
Thus we have reduced to the problem of approximating $u_1\approx f_1$.

(2)
Replacing $g$ by $g_1$,
it follows from step (1) that we may assume that $\supp g\subset Q_0\in\mD$ and $\int_{Q_0}g=0$.
Denote by $u$ the dyadic average extension of $g$, and write $u_Q:= \fint_Q g(y) dy$.
We construct the approximant $f$ using the following stopping time argument.
Given any cube $Q\in\mD$, we define the stopping cubes
$$
  \omega(Q):= \{\text{maximal } R\in\mD \text{ such that } R\subset Q \text{ and } |u_R-u_Q|\ge \epsilon M_\mD g(R) \}.
$$
We then define generations of stopping cubes under $Q_0$ inductively as follows.
\begin{gather*}
  \omega_0:= \{Q_0\}, \qquad \omega_1:= \omega(Q_0), \\
  \omega_{k+1}:= \bigcup_{Q\in \omega_k} \omega(Q), \qquad k=1, 2, \ldots \\
  \omega_*:= \bigcup_{k=0}^\infty \omega_k.
\end{gather*}
Furthermore, for $Q\in \omega_*$ we define the ``dyadic sawtooth'' region
$$
  \Omega(Q):= \widehat Q\setminus \clos{\bigcup_{R\in \omega(Q)}\widehat R}\subset \widehat Q.
$$

We define $f$ to be the locally constant function in $\R^{1+n}_+$ which takes the value $u_Q$ on $\Omega(Q)$ for each $Q\in\omega_*$, i.e.,
$$
  f_R:= u_Q, \qquad \text{when } W_R\subset \Omega(Q), Q\in \omega_*,
$$
and $f=0$ on $\R^{1+n}_+\setminus \widehat Q_0$.
From this construction it is clear that $f$ has non-tangential limits almost everywhere.
To verify that $\|N(f-u)\|_p\le \epsilon \|g\|_p$, we note directly from the stopping condition that
$$
 N_\mD(f-u)(x)= \sup_{Q\ni x, Q\in \mD} |f_Q-u_Q|\le \epsilon  \sup_{x\in Q\in \mD} M_\mD g(Q)= \epsilon M_\mD g(x),
$$
from which the estimate follows.

(3)
We next establish the main estimate, namely that of $C(\pd_t f)$.
We fix $Q_1\in \mD$ with $Q_1\subset Q_0$ and estimate
$$
 \sum_{Q\in \omega_*, Q\subset Q_1} |u_Q-u_{Q_*}| |Q|\le \frac 1\epsilon\sum_{Q\in \omega_*, Q\subset Q_1} |u_Q-u_{Q_*}|^2 \frac{|Q|}{M_\mD g(Q)},
$$
where we write $Q_*$ for the stopping parent of $Q$, that is the smallest $Q_*\in \omega_*$ such that $Q_*\supsetneqq Q$,
and exceptionally $(Q_0)_*:= Q_0$.

Define the square function
$$
  S g(x):= \left( \sum_{Q\in \omega_*, Q\subset Q_1} |\ave{g}_Q-\ave{g}_{Q_*}|^2 1_Q(x) \right)^{1/2},\qquad
  \ave{g}_Q:=\fint_Q g(y)dy.
$$
Recalling that $u_Q=\ave{g}_Q$ for $Q\in\omega_*$, Lemma~\ref{lem:truncmax} gives
\begin{multline*}
\sum_{Q\in \omega_*, Q\subset Q_1} |u_Q-u_{Q_*}|^2 \frac{|Q|}{M_\mD g(Q)} \\
\lesssim \sum_{Q\in \omega_*, Q\subset Q_1} |\ave{g}_Q-\ave{g}_{Q_*}|^2\int_{Q_1} 1_{Q}(x)\frac {dx}{M_\mD g(x)}=
\int_{Q_1}|S g(x)|^2 \frac {dx}{M_\mD g(x)}.
\end{multline*}
We now use some properties of (dyadic versions of) the Muckenhoupt weight classes $A_p$; these are easy variants of well-known results for the usual $A_p$ classes, found e.g. in \cite{Duo}, Chapter 7.
Writing
$$
  (M_\mD g)^{-1}= 1\cdot ((M_\mD g)^\gamma)^{1-q},
$$
for some $\gamma\in (0,1)$ and $q= 1+1/\gamma\in (2,\infty)$, it follows (cf. \cite{Duo}, Theorem~7.7(1) and Proposition~7.2(3)) that $(M_\mD g)^\gamma\in A_1(dx)$ and $(M_\mD g)^{-1}\in A_q(dx)\subset A_\infty(dx)$, with $A_q$ constants independent of $g$.

We now apply Proposition~\ref{prop:stoppedweightedSF}, with the collection of cubes
$\tilde \omega_*:= \sett{Q\in\omega_*}{Q\subset Q_1}$, the function
$$
  \tilde g:=
  \begin{cases}
    g(x)- \fint_{Q_1} g dx, & \qquad x\in Q_1, \\
    0, &\qquad x\notin Q_1,
  \end{cases}
$$
the weight $w:= (M_\mD g)^{-1}$ and $p=2$.
This gives
$$
  \int_{\R^n} |S_{\tilde \omega_*}\tilde g |^2 dw\lesssim \int_{\R^n} |M_\mD \tilde g|^2 dw= \int_{Q_1} |M_\mD \tilde g|^2 dw\lesssim \int_{Q_1}|M_\mD g|^2 dw.
$$
Thus
\begin{multline*}
  \int_{Q_1}|S g(x)|^2 dw\lesssim  \int_{Q_1}|S_{\tilde \omega_*} g(x)|^2 dw + \int_{Q_1}|M_\mD g|^2 dw \\ \le
  \int_{\R^n}|S_{\tilde \omega_*} \tilde g(x)|^2 dw + \int_{Q_1}|M_\mD g|^2 dw\lesssim \int_{Q_1}|M_\mD g|^2 dw
  \le |Q_1|\inf_{Q_1} M_\mD(M_\mD g),
\end{multline*}
and so
$$
  \|C(\pd_t f)\|_p \lesssim \epsilon^{-1} \| M_\mD(M_\mD g) \|_p \lesssim\epsilon^{-1} \|g\|_p.
$$

(4)
To complete the proof, we use Lemma~\ref{lem:tanggradest} and obtain the Carleson estimate
$$
  \|C(\nabla f)\|_p\approx \|C(\pd_t f)\|_p+\|C(\nabla_x f)\|_p\lesssim
  \|C(\pd_t f)\|_p+ \|M_\mD g\|_p\lesssim \|g\|_p.
$$
\end{proof}

\section{Application to harmonic measure}   \label{sec:harmmeas}

Before proving Theorem~\ref{thm:3}, we discuss in this section an important application to the solvability of the Dirichlet boundary value problem, of such $L_p$ approximability for solutions to an elliptic equation.
As noted in the introduction, in the end point case $p=\infty$, there is a well known equivalence 
between 
\begin{enumerate}
\item comparability of non-tangential maximal functions  and square functions for solutions, \label{pr1}
\item approximability of solutions by functions of bounded variation, and \label{pr2}
\item $A_\infty$ control of harmonic measure, \label{pr3}
\end{enumerate}
for a given real elliptic divergence form equation. 
It is important to note that this equivalence holds for all equations with real and possibly non-symmetric coefficents, including those which depend on the transversal direction $t$. There are known examples by Caffarelli, Fabes and Kenig~\cite{CaFK}, of symmetric such coefficients for which harmonic measure is not $A_\infty$, and therefore the approximability and comparability properties may fail as well for $t$-dependent coefficients.

In this section, the goal is to demonstrate that in the above equivalences, we may replace \eqref{pr2} by the following local version of the conclusion in  Theorem~\ref{thm:3}.
\begin{itemize}
\item[$(2p)$]  For each $0<\epsilon<1$, there exists $c_\epsilon'<\infty$ such that for every weak solution $u:\R^{1+n}_+\to \R$ to $\divv_{t,x} A(t,x)\nabla_{t,x} u(t,x)=0$
with $\|u\|_{L_\infty(\R^{1+n}_+)}\le 1$, and every cube $Q\subset \R^n$,
there exists a function $f_Q$ in $\R^{1+n}_+$ of locally bounded variation with estimates
$$
\begin{cases}
    \|N_\ell(f_Q-u)\|_{L_p(Q)}\le \epsilon |Q|^{1/p}, \\
    \|A_\ell(\nabla f_Q)\|_{L_p(Q)}\leq c_{\epsilon}  |Q|^{1/p}.
\end{cases}
$$
Here $N_\ell$ and $A_\ell$ denote versions of the non-tangential maximal and area functionals 
from Definition~\ref{defn:contfuncNCA} using cones 
$\sett{(y,t)}{|y-x|\le \alpha t, t<\ell}$ truncated 
at height $\ell= \ell(Q)$.
\end{itemize}

Note that we have used the truncated area functional $A_\ell$ in the second estimate in $(2p)$, in contrast to the Carleson functional $C$ that we used in the global version in Thm. \ref{thm:3}. This, however, is inessential by Prop. \ref{prop:AGequiv}, which also easily extends to the truncated situation by routine modifications. The chosen formulation of property $(2p)$ is motivated by its application to the harmonic measure below, where the area functional leads to the most immediate connection.

Assuming $A_\infty$ control of the harmonic measure, comparability of non-tangential maximal functions  and square functions for solutions follows by \cite{DJK}. Given such $N\approx S$ comparability, approximability follows, both in the case $p=\infty$ as in \cite{KKPT, HKMP}, and for $1<p<\infty$ as shown in Section~\ref{sec:solutions} of this paper.
Also the local approximability $(2p)$ follows from $N\approx S$ comparability, since 
our estimates are derived from pointwise estimates.
In general, without assuming such comparability, we note the following.

\begin{prop}
Consider a possibly $t$-dependent real equation
\begin{equation*}
   \divv_{t,x} A(t,x)\nabla_{t,x} u(t,x)=0. 
\end{equation*}
If the approximability property for solutions in Theorem~\ref{thm:3} holds and if
\begin{equation*}
 n/(n-1)\le p<\infty, 
\end{equation*}
then the local approximability property $(2p)$ also holds.  
\end{prop}

\begin{proof}
  Let $u$ be a solution with properties as in $(2p)$.
  Given a cube $Q\subset \R^n$, write $u= u_0+u_1$, 
  where $u_0$ is the solution to the Dirichlet problem with boundary data $\eta_Qu|_{\R^n}$,
  where $\eta_Q=1$ on $5Q$ and supported on $6Q$. 
  From the maximum principle and \cite[Lem. 4.9]{HKMP}, we have the estimate
  $$
    \abs{u_0(t,x)}\lesssim \min\Big(1, \Big[\frac{\ell(Q)}{|(t,x)-(0,x_Q)|}\Big]^{n-1+\nu}\Big), \qquad (t,x)\in\R^{1+n}_+, 
  $$
  for some $\nu>0$, where $x_Q$ denotes the center of $Q$.
  In particular, $Nu_0(x)\lesssim \min(1, (\ell(Q)/|x-x_Q|)^{n-1+\nu})$ for $x\in\R^n$, and thus
   $\|Nu_0\|_{L_p(\R^n)}\lesssim |Q|^{1/p}$ if $p\ge n/(n-1)$.
  
  Let $f_Q:= f_0+u_1$, where $f_0$ is the approximant to $u_0$ given by the assumed global approximability, so that $f_Q-u=f_0-u_0$.
We will show that this $f_Q$ qualifies for (2p). For the first estimate, this is immediate from the assumed properties of the global approximant and the observations just made, namely
\begin{equation*}
  \|N_\ell(f_Q-u)\|_{L_p(Q)}
  \leq\|N(f_0-u_0)\|_{L_p(\R^n)}
  \leq\epsilon\|Nu_0\|_{L_p(\R^n)}
  \lesssim\epsilon\abs{Q}^{1/p}.
\end{equation*}
For the second estimate, we separately consider the two terms $f_0$ and $u_1$. First,
\begin{equation*}
  \|A_\ell(\nabla f_0)\|_{L_p(Q)}
  \lesssim\|C(\nabla f_0)\|_{L_p(\R^n)}
  \lesssim c_\epsilon \|Nu_0\|_{L_p(\R^n)}
  \lesssim c_\epsilon \abs{Q}^{1/p}.
\end{equation*}
Finally, we turn to $u_1$. Let $R\subset\R^n$ be a cube with $\max(\ell(R),\dist(R,Q))\leq\ell:=\ell(Q)$. Since $u_1$ is a bounded solution with vanishing boundary values on $5Q\supset 2R$, we may apply the boundary Cacciopoli estimate (see for example \cite[(1.3B)]{KKPT}) to deduce that
\begin{equation*}
  \fint_{\widehat R}\abs{\nabla u_1}\lesssim
  \Big(\fint_{\widehat R}\abs{\nabla u_1}^2\Big)^{1/2}
  \lesssim\frac{1}{\ell(R)}\Big(\fint_{\widehat{2R}}\abs{u_1}^2\Big)^{1/2}
  \lesssim\frac{1}{\ell(R)},
\end{equation*}
which shows that
\begin{equation*}
  C_{\ell}(\nabla u_1)(x)=\sup_{\substack{R\owns x \\ \ell(R)\leq\ell}}\frac{1}{\abs{R}}\int_{\widehat R}\abs{\nabla u_1}\lesssim 1
\end{equation*}
in a neighbourhood of $Q$. By the $L_p$-comparability of $C_\ell$ and $A_\ell$ (a routine modification Prop. \ref{prop:AGequiv}), this gives
\begin{equation*}
  \| A_\ell(\nabla u_1)\|_{L_p(Q)}\lesssim \abs{Q}^{1/p}
\end{equation*}
and completes the verification of $(2p)$.  
\end{proof}

We now consider the main result in this section, namely that $(2p)$ implies \eqref{pr3}.

\begin{thm}
   Let $1<p<\infty$ and consider a real equation $\divv_{t,x} A(t,x)\nabla_{t,x} u(t,x)=0$. If the local approximability property $(2p)$ holds, then harmonic measure belongs
   to $A_\infty$. 
\end{thm} 

\begin{proof}
Our proof is an adaption of the proof of \cite[Thm. 2.3]{KKPT}, the case $p=\infty$, and we 
only point out the changes needed for $p<\infty$.
Fix a solution $u$ to $\divv_{t,x} A(t,x)\nabla_{t,x} u(t,x)=0$ with $\|u\|_{L_\infty(\R^{1+n}_+)}\le 1$.
Following \cite{KKPT}, but not their notation, we consider the counting function 
\begin{equation*}
\begin{split}
  K_r(x)
  :=\max\Big\{k &: \exists\text{ $k$ points } z_i=(x_i,t_i)\in \Gamma_r(x)\text{ such that} \\
    &\phantom{:}\ t_i<\theta t_{i-1}\text{ and }\abs{u(z_i)-u(z_{i-1})}\geq\epsilon_0\Big\},
\end{split}
\end{equation*}
where $\epsilon_0,\theta\in(0,1)$ are parameters, and 
\begin{equation*}
  \Gamma_r(x):=\sett{(y,t)}{|y-x|\le \alpha_0 t, t<r}
\end{equation*}
is a truncated cone based at $x$. In the proof of \cite[Thm. 2.3]{KKPT}, the classical $\epsilon$-approximability property is only used through the following consequence established in  \cite[Lem. 2.9]{KKPT}:
\begin{equation*}
  \fint_Q K_{\ell(Q)}(x) dx\le c(\epsilon_0,\theta).
\end{equation*}
Thus, it suffices to establish the same conclusion under our approximation property $(2p)$. We will in fact show that
\begin{equation}   \label{eq:countest}
  \fint_Q K_{\ell(Q)}(x)^p dx \le C(\epsilon_0, \theta),
\end{equation}
from which the earlier estimate follows by H\"older's inequality.

To prove \eqref{eq:countest}, we fix a cube $Q$, pick an approximant $f_Q$ given by the 
hypothesis $(2p)$, and note the estimate
$$
  |Q_b|
  := \abs{\sett{y\in Q}{N_\ell(f_Q-u)>C_1\epsilon}}
  \le (C_1\epsilon)^{-p}\|N_\ell(f_Q-u)\|_{L_p(Q)}^p
  \le |Q|/C_1^p.
$$
Let $\mathcal W$ denote a Whitney covering of $Q_b$ by cubes $R\subset Q$.
For $x\in Q$, 
we note that the pointwise estimate
$$
   |u-f_Q|\le C_1 \epsilon
$$
holds in $\Gamma_{\ell(Q)}(x)\setminus \bigcup_{R\in\mathcal W} \widehat{C_2R}$ for some $C_2<\infty$, provided that $\alpha$ appearing in $(2p)$ is chosen large enough
depending on $\alpha_0$.  Let $R_x$ be a largest cube $R\in\mathcal W$ such that $\widehat{C_2R}$ intersects $\Gamma_{\ell(Q)}(x)$. Then $x\in C_2' R_x$ for some $C_2'\ge C_2$, and the pointwise estimate above holds throughout $\Gamma_{\ell(Q)}(x)\setminus\Gamma_{\ell(C_2 R_x)}(x)$.

Now, let $z_i=(t_i,y_i)\in\Gamma_{\ell(Q)}(x)$, with $i=1,\ldots,k$, be points as in the definition of $K_r(x)$, and let $h$ be the largest index such that $t_h>C_2\ell(R_x)$. It follows from interior H\"older regularity that a jump estimate $\abs{u(w_i)-u(w_{i-1})}\geq\frac34\epsilon_0$ persists for all $w_j\in B(y_j,\eta t_j)\times\{t_j\}$ and a suitably small $\eta$. If $C_1\epsilon\le\epsilon_0/4$, it follows that also $\abs{f_Q(w_i)-f_Q(w_{i-1})}\geq\frac14\epsilon_0$ for $w_j\in (B(y_j,\eta t_j)\times\{t_j\})\cap\Gamma_r(x)$ and $i\leq h$. Estimating the difference $\abs{f_Q(w_i)-f_Q(w_{i-1})}$ by an integral of $\nabla f_Q$ over the connecting line, and integrating averaging over $w_j\in (B(y_j,\eta t_j)\times\{t_j\})\cap\Gamma_r(x)$ for $j=i,i-1$, it follows that,
\begin{equation*}
  \tfrac14\epsilon_0\lesssim\int_{\Gamma_{t_{i-1}}(x)\setminus\Gamma_{t_i}(x)}\abs{\nabla f_Q(t,y)}t^{-n}dt\,dy
\end{equation*}
and summing over $i=2,\ldots,h$, that
\begin{equation*}
  (h-1)\lesssim A_{\ell(Q)}(\nabla f_Q)(x).
\end{equation*}
On the other hand, the remaining points $z_i$, with $i=h+1,\ldots,k$, all belong to $\Gamma_{\ell(C_2 R_x)}(x)$, so that $k-h\leq K_{\ell(C_2 R_x)}(x)$, by definition. So, altogether, we have
\begin{equation*}
  k=1+(h-1)+(k-h)\leq 1+C_3A_{\ell(Q)}(\nabla f_Q)(x)+K_{\ell(C_2 R_x)}(x).
\end{equation*}
Recalling that $K_{\ell(Q)}(x)$ is the maximal value of such numbers $k$, and that $C_2' R_x\owns x$, we arrive at
\begin{equation*}
  K_{\ell(Q)}\leq 1+C_3 A_{\ell(Q)}(\nabla f_Q)+\sup_{R\in\mathcal W} 1_{C_2' R}K_{\ell(C_2' R)},
\end{equation*}
where we also estimated $C_2\leq C_2'$.

%
%

To prove \eqref{eq:countest}, we set 
$$
  D:= \sup_Q \fint_Q K_{\ell(Q)}(x)^p dx
$$
and integrate 
\begin{equation*}
\begin{split}
 K_{\ell(Q)}(x)^p &\leq 3^{p-1} (1+C_3^p A_{\ell(Q)}(\nabla f_Q)^p+ \sup_{R\in\mathcal W} 1_{C_2' R}(x)K_{\ell(C_2' R)}(x)^p) \\
&\leq 3^{p-1} (1+C_3^p A_{\ell(Q)}(\nabla f_Q)^p+ \sum_{R\in\mathcal W} 1_{C_2' R}(x)K_{\ell(C_2' R)}(x)^p),
\end{split}
\end{equation*}
to get
$$
  \fint_Q K_{\ell(Q)}(x)^p dx\le
    3^{p-1}(1+ C_3^p c_\epsilon^p + |Q|^{-1} (C_2')^n|Q_b| D ).
$$
Choosing $C_1$ large (and then $\epsilon$ small) and taking supremum over $Q$, we can hide the second term on the right hand side, on the left hand side.
To guarantee the finiteness of $D$ in the first place, one may initially replace $K_r(x)$ by $\min(K_r(x),M)$ and pass to the monotone limit $M\to\infty$ in the end.)
This proves \eqref{eq:countest}, which concludes the proof.
\end{proof}

\section{Approximation of solutions to elliptic equations}  \label{sec:solutions}

In this section, we prove Theorem~\ref{thm:3}, but first fix notation.
We only consider dyadic cubes of every $N$th generation for some fixed $N$, and refer to these simply as ``dyadic cubes'', notation 
$$
\mD^\delta:=\sett{Q\in\mD}{\ell(Q)=2^{-kN}, k\in\Z},
$$
where we write $\delta:=2^{-N}$ for the change of scale between consecutive generations of these dyadic cubes. The main advantage of this consideration is to gain a convenient control of boundary effects: For positive integers $m$, we have the dichotomy that each dyadic subcube $Q'\subsetneqq Q$ is either a ``centred cube'', meaning
$$
  Q'\subset (1-2m\delta)Q,
$$
or a ``boundary cube'', meaning
$$
 Q'\subset Q\setminus (1-2m\delta)Q.
$$
As long as $2m\delta<1$, both these consist of a positive fraction of the total volume of $Q$, but this could never be achieved with $\delta=2^{-1}$ when every dyadic child is a boundary cube.

The $\mD^\delta$-dyadic versions of the functionals $M$, $N$ and $C$ are denoted by
$M_{\mD^\delta}$, $N_{\mD^\delta}$ and $C_{\mD^\delta}$.
One verifies that the estimates analogous to Proposition~\ref{prop:dyadicvsnondya} hold.
We also use the following notation for $Q\in\mD^\delta$, where the parameter $\eta>0$ will be eventually chosen small relative to the given $\epsilon$ appearing in the statement of Theorem~\ref{thm:3}.
(The parameter $\delta=2^{-N}$ will also be chosen small, but independent of $\epsilon$.)
\begin{itemize}
  \item $\widehat Q:=(0,\ell(Q))\times Q$ is the Carleson box, as before.
  \item $p_Q:=((1-\eta)\ell(Q),c_Q)$ is the ``corkscrew point'' of $\widehat Q$, where $c_Q$ denotes the center of $Q$.
  \item $\widetilde{Q}:=\{\ell(Q)\}\times \eta Q$ is a small hypersurface on the top boundary of $\widehat Q$, around the centre.
  \item $W_Q:=\widehat Q\setminus\bigcup_{\mD^\delta\ni Q'\subsetneqq Q}\widehat Q'=[\delta\ell(Q),\ell(Q))\times Q$ is a Whitney-type rectangle.
  \item $\Gamma_Q:=\{(t,x):t>\delta\ell(Q)+\dist(x,Q)\}$ is an epigraph domain containing $W_Q$.
\end{itemize}

An outline of the proof of Theorem~\ref{thm:3} is as follows.
We construct the approximant $f$ as follows. 
We define a family of ``stopping cubes'' $\mathcal{S}$ in \eqref{eq:Sstop}, and the corresponding sawtooth regions $\Omega_{\mathcal S}(S):=\bigcup_{Q:\pi_{\mathcal S}Q=S}W_Q$, where $\pi_{\mathcal S}Q=S$ means that $S$ is the smallest stopping cube such that $S\supset Q$; the family $\mathcal{S}$ is built in such a way that the value of $u(p_Q)$ varies relatively little among all $Q$ with $\pi_{\mathcal S}Q=S$. The first approximation to $u$ is then given by $\varphi_1:=\sum_{S\in\mathcal{S}}u(p_S)\cdot 1_{\Omega_{\mathcal{S}}(S)}$.

However, this approximation fails to be good on the Whitney regions $W_R$, where the oscillation of $u$ is relatively large, more precisely, when it happens that
\begin{equation}   \label{eq:defnfamR}
  \osc_{W_R}u:=\sup_{z,w\in W_R}\abs{u(z)-u(w)}>\epsilon M_{\mD^\delta}(Nu)(R):=\epsilon \sup_{Q\supset R}\fint_Q Nu(x) dx.
\end{equation}
Note that here the defining condition is simpler than the stopping conditions considered above, in that it can be directly checked for any cube, without reference to the previously chosen members of the stopping family. We label the family of the cubes $R$ in \eqref{eq:defnfamR} by $\mathcal R$, and introduce the additional correction $\varphi_2:=\sum_{R\in\mathcal{R}}(u-\varphi_1)\cdot 1_{W_R}$. The final approximation is then given by $f:=\varphi_1+\varphi_2$. The verification of $N_{\mathcal{D}^\delta}(f-u)\lesssim\epsilon Nu$ will then be straightforward from the construction of the collections $\mathcal{S}$ and $\mathcal{R}$. The pointwise estimate for the Carleson functional, $C_{\mathcal{D}^\delta}(\nabla f)\lesssim_\epsilon M_{\mathcal{D}^\delta}(Nu)$, is verified separately for $\varphi_1$ and $\varphi_2$ in place of $f$; these bounds depend in particular on the Carleson property of both $\mathcal{S}$ and $\mathcal{R}$, established in Lemmas~\ref{lem:CarlesonS} and!
  \ref{lem:CarlesonR}, and the estimates are completed in Lemmas~\ref{lem:phi1} and \ref{lem:phi2}. 

Lemmas~\ref{lem:CarlesonS} and \ref{lem:CarlesonR} ultimately build on two estimates which we borrow from \cite{HKMP}.
These are
\begin{equation}   \label{est:N<S}
  \int_{\theta Q} |u(\psi(x),x)-u(p_Q)|^2 dx\lesssim \iint_{(t,x)\in \widehat Q, t>\psi(x)}|\nabla u(t,x)|^2(t-\psi(x)) dtdx,
\end{equation}
and
\begin{equation}   \label{est:S<N}
 \frac 1{|Q|}\iint_{(t,x)\in\widehat Q, t>\psi(x)}|\nabla u(t,x)|^2(t-\psi(x)) dtdx\lesssim \sup_{\substack{(t,x)\in\R^{1+n}_+ \\  t>\psi(x) }}|u(t,x)|^2,
\end{equation}
for weak solutions $u$ to an elliptic equation $Lu=0$ as in Theorem~\ref{thm:3}.
Here $\theta\in(0,1)$, $p_Q$ is a ``cork-screw point'' above $Q$ in the Carleson box $\widehat Q$, and
$\psi\ge 0$ is a Lipschitz function. Note that the implicit constants in the two estimates depend on the ellipticity constants $\lambda_A$ and $\|A\|_\infty$ from Theorem~\ref{thm:3} and on $\|\nabla \psi\|_\infty$ and dimension, but not otherwise on $A$, $\psi$, $u$ or $Q$.
The first estimate \eqref{est:N<S} follows from \cite[Cor. 1.17]{HKMP} upon replacing $u$ by
$u-u(p_Q)$ and using interior regularity and Poincar\'e's inequality to remove the error term.
The second estimate \eqref{est:S<N} follows from \cite[Cor. 1.10]{HKMP} upon pulling back that result from the half space
to the epigraph domain $t>\psi(x)$.

The construction in this section builds on that in the case $p=\infty$ from \cite{Gar, D3, KKPT, HKMP}, but with non-trivial modifications.
The construction \eqref{eq:defnfamR} of the family $\mathcal{R}$ of Whitney regions with large oscillation of $u$,
goes back to \cite{Gar}, as does the stopping construction \eqref{eq:Sstop}.
The main novelty here is that for $L_p$, $p<\infty$, we require a variable threshold in these constructions,
expressed in terms of the maximal function of $u$.
This requires a second parallel stopping construction \eqref{eq:princstopping}, which has the effect of freezing this threshold.
Such multiple stopping time constructions have appeared earlier in \cite{ARR, LNbook, R2}.

Finally, to pass from dyadic sawtooths to Lipschitz sawtooths to be able to use the above estimates \eqref{est:N<S} and \eqref{est:S<N}, we follow the construction in 
\cite{HKMP} in Lemma~\ref{lem:sortOfVitali}. We now turn to the details.

\subsection{Construction of stopping cubes}  \label{subsec:stoppingconstr}

We start with some generalities.
Let $\mathcal{C}(Q',Q)\in\{\operatorname{true},\operatorname{false}\}$ be some ``criterion'' that assigns a truth value to every pair of (dyadic) cubes $Q'\subset Q$. We specifically agree that $\mathcal{C}(Q,Q)=\operatorname{false}$ for every cube $Q$.
By the ``stopping family'' with initial collection $\mathcal{I}\subset\mD^\delta$ and stopping criterion $\mathcal{C}$ we understand the family of dyadic cubes $\mathcal{F}=\mathcal{F}(\mathcal I,\mathcal C)$ constructed as follows: We initialize $\mathcal{F}:=\mathcal{I}$. Then we add to $\mathcal{F}$ all $F'\in\mD^\delta\setminus\mathcal{F}$ such that
\begin{enumerate}
  \item[{\rm (a)}]  $\mathcal{C}(F',F)$ is true for some $F\in\mathcal{F}$ with $F'\subsetneqq F$, and
  \item[{\rm (b)}] $F'$ is not contained in any $F''\subsetneqq F$ with either $F''\in\mathcal{F}$ or $\mathcal{C}(F'',F)$ is true.
\end{enumerate}
We repeat this addition indefinitely. This is seen to yield a well defined family $\mathcal{F}\subset\mD^\delta$.

For every $Q\in\mD^\delta$, let $\pi_{\mathcal F}Q$ denote the minimal $F\in\mathcal{F}$ such that $Q\subset F$, where the possibility that $F=Q$ is not excluded. The stopping family with initial collection $\mathcal{I}$ and stopping criterion $\mathcal{C}$ has the property that $\mathcal{C}(Q,F)$ is false
 whenever $\pi_{\mathcal F}Q=F$; namely, the latter means by definition that there \emph{does not exist} any intermediate stopping cube $F'$ with $Q\subset F'\subsetneq F$, thus all intermediate cubes $Q'$ with $Q\subset Q'\subsetneq F$ do \emph{not} satisfy the stopping condition $\mathcal{C}(Q',F)$, hence $\mathcal{C}(Q',F)$ is false for all these $Q'$, and in particular for $Q'=Q$ itself. For $F\in\mathcal F$, we denote by $\operatorname{ch}_{\mathcal{F}}(F):=\{F'\in\mathcal{F}\text{ maximal: }F'\subsetneqq F\}$ the family of its $\mathcal F$-\emph{children}.

A stopping criterion $\mathcal{C}$ is \emph{sparse}, if
\begin{equation*}
  \sum_{\substack{Q'\subsetneqq Q\text{ maximal:}\\ \mathcal{C}(Q',Q)\text{ is true} }}\abs{Q'}\leq \tau\abs{Q}
\end{equation*}
for some fixed $\tau<1$. It is straightforward to check that if the initial collection $\mathcal{I}$ is Carleson, that is
$$
  \sup_{Q\in\mathcal{I}}\frac{1}{\abs{Q}} \sum_{\mathcal{I}\ni R\subset Q} |R|<\infty,
$$
 and if the stopping criterion $\mathcal{C}$ is sparse, then the stopping family $\mathcal{F}=\mathcal{F}(\mathcal I,\mathcal C)$ is Carleson.

A priori, the stopping collections produced by a dyadic algorithm may not be so well behaved geometrically. This is to some extent remedied by the following lemma, which builds on ideas from \cite[Sec. 5]{HKMP}.

\begin{lemma}\label{lem:sortOfVitali}
Consider a cube $Q$ and a disjoint collection $\mathcal{Q}$ of its dyadic subcubes. 
We say that $Q'\in\mathcal{Q}$ is
\begin{itemize}
  \item  \emph{centred (in $Q$)} if $Q'\subset(1-2\delta)Q$, and
  \item \emph{uncovered (by $\mathcal{Q}$)} if $(\ell(Q'')-\ell(Q') )/\dist(Q',Q'')\leq \delta^{-1}$ for all $Q''\in\mathcal{Q}$.
\end{itemize}
(See Figure~\ref{fig:covered} below for an illustration of covered and uncovered cubes.)
Let
\begin{equation*}
   \mathcal{Q}^*:=\{Q'\in\mathcal{Q}:Q'\text{ is centred and uncovered}\}.
\end{equation*}
For some constants $\tau\in(0,1)$ and $C$, we then have
\begin{equation*}
  \sum_{Q'\in\mathcal{Q}}\abs{Q'}
  \leq\tau\abs{Q}+C\sum_{Q''\in\mathcal{Q}^*}\abs{Q''}.
\end{equation*}
\end{lemma}

\begin{proof}
If $Q'$ is not uncovered, then $\ell(Q'')-\delta^{-1}\dist(Q',Q'')>\ell(Q')$ for some $Q''\in\mathcal{Q}$, and we say that this $Q''$ \emph{covers} $Q'$. Then in particular $\ell(Q')<\ell(Q'')$, and hence $\ell(Q')\leq\delta\ell(Q'')$, but also $\dist(Q',Q'')<\delta\ell(Q'')$.

Further, if $Q'$ is not uncovered, it is covered by some $Q_1$ which, if not uncovered, is covered by some $Q_2$, and so on. Since $\ell(Q_k)$ increases geometrically and is bounded by $\ell(Q)$, the chain must terminate after finitely many steps with some uncovered $Q_k$. The $\ell^\infty$-distance of the furthest point of $Q'=:Q_0$ from the centre of $Q_k$ can be at most
\begin{equation*}
\begin{split}
  \sum_{j=0}^{k-1}[\ell(Q_j)+\dist(Q_j,Q_{j+1})]+\frac12\ell(Q_k)
  &\leq\sum_{j=0}^{k-1}2\delta^{k-j}\ell(Q_k)+\frac12\ell(Q_k) \\
  &<\Big(\frac{2\delta}{1-\delta}+
  \frac12\Big)\ell(Q_k)\leq\frac{5}{2}\ell(Q_k),
\end{split}
\end{equation*}
since $\delta\leq\frac12$,
and hence $Q'\subset 5Q_k$. 
We have
\begin{equation*}
  \sum_{Q'\in\mathcal{Q}}\abs{Q'}
  =\sum_{\substack{Q'\in\mathcal{Q} \\ Q'\subset(1-2m\delta)Q}}\abs{Q'}+\sum_{\substack{Q'\in\mathcal{Q} \\ Q'\subset Q\setminus(1-2m\delta)Q}}\abs{Q'}=:I+II,
\end{equation*}
where $II\leq(1-(1-2m\delta)^n)\abs{Q}$ by disjointness. On the other hand, every $Q'$ appearing in $I$ is contained in $5Q''$ for some uncovered $Q''\in\mathcal{Q}$. In particular, $5Q''$ intersects with $Q'\subset(1-2m\delta)Q$. Thus, the $\ell^\infty$-distance of the furthest point of $5Q''$ from the centre of $Q$ is at most
\begin{equation*}
  (1-2m\delta)\ell(Q)/2 +5\ell(Q'') \leq (5\delta+\frac12-m\delta)\ell(Q),
\end{equation*}
and hence $Q''\subset 5Q''\subset (1-2(m-5)\delta)Q$. Since all $Q'$ in $I$ are covered by such $5Q''$, we have
\begin{equation*}
  I\leq\sum_{\substack{Q''\text{ uncovered} \\ Q''\subset(1-2(m-5)\delta)Q}}\abs{5Q''}
  \leq 5^n\sum_{Q''\in\mathcal{Q}^*}\abs{Q''},
\end{equation*}
provided that we take $m\geq 6$, and also $\delta<1/(2m)$ for term $II$.
\end{proof}

\bigskip

Let us now fix as our initial collection $\mathcal{I}$ some increasing chain of cubes $I_0\subsetneqq I_1\subsetneqq I_2\subsetneqq\cdots$ that exhaust $\R^n$. Clearly this is Carleson.


We define the ``principal cubes'' $\mathcal{P}$ as the stopping family with initial collection $\mathcal{I}$ and the stopping criterion $\mathcal{C}(Q',Q)$ given by
\begin{equation}   \label{eq:princstopping}
  M_{\mD^\delta}(Nu)(Q')=\sup_{R\supseteq Q'}\fint_{R}Nu(x) dx>A\cdot M_{\mD^\delta}(Nu)(Q),
\end{equation}
for some fixed $A>1$ to be chosen.
To verify that this criterion is sparse,
and therefore $\mathcal{P}$ is Carleson, select a disjoint family of subcubes $Q'\subset Q$ that satisfy \eqref{eq:princstopping}. Then
\begin{equation*}
\begin{split}
  \sum\abs{Q'}
  &\leq\abs{\{M_{\mathcal{D}^\delta}(1_{Q}Nu)>A\cdot M_{\mathcal{D}^\delta}(Nu)(Q)\}} \\
  &\leq\frac{1}{A\cdot M_{\mathcal{D}^\delta}(Nu)(Q)}\int_Q Nu(x)dx=\frac{\abs{Q}}{A}
\end{split}
\end{equation*}
by the weak-type $(1,1)$ estimate for the maximal operator $M_{\mathcal{D}^\delta}$.

The usefulness of the numbers $M_{\mD^\delta}(Nu)(Q)$ lies in the fact that they control the values of $u$ in the entire graph-domain $\Gamma_Q$.

\begin{lemma}\label{lem:uControlByMNu}
We have the estimate
\begin{equation*}
  \sup_{\Gamma_Q}\abs{u}\leq M_{\mD^\delta}(Nu)(Q).
\end{equation*}
\end{lemma}

\begin{proof}
It is enough to observe that $\Gamma_Q\subset \Gamma_{1/\delta}(x):=\{(t,y):\abs{y-x}<\delta^{-1}t\}$ for any $x\in Q$, and therefore
\begin{equation*}
  \sup_{\Gamma_Q}\abs{u}\leq\inf_{x\in Q}\sup_{\Gamma_{1/\delta}(x)}\abs{u}
  \leq\inf_{x\in Q}Nu(x)\leq \fint_Q Nu dx \leq M_{\mD^\delta}(Nu)(Q),
\end{equation*}
provided that the aperture defining $N$ is at least $\delta^{-1}$.
\end{proof}

Finally, we define the ``stopping cubes'' $\mathcal{S}$ as the stopping family with initial collection $\mathcal{P}$ and the stopping criterion $\mathcal{C}(Q',Q)$ given by
\begin{equation}\label{eq:Sstop}
  \abs{u(p_{Q'})-u(p_Q)}>\epsilon M_{\mD^\delta}(Nu)(Q').
\end{equation}
We observe the following self-improvement of this criterion.

\begin{lemma}\label{lem:SstopBootstrap}
Under the condition \eqref{eq:Sstop}, we have
\begin{equation*}
  \abs{u(z)-u(p_Q)}\gtrsim\epsilon M_{\mD^\delta}(Nu)(Q')\qquad
  \text{for all } z\in \widetilde{Q}'=\{\ell(Q')\}\times \eta Q',
\end{equation*}
provided that $\eta$ satisfies $\eta^\alpha\ll\epsilon$, where $\alpha>0$ is the H\"older exponent from interior regularity estimates for $u$.
\end{lemma}

\begin{proof}
If $z\in\widetilde{Q'}:=\{\ell(Q')\}\times\eta Q'$, the interior regularity of solutions $u$ to $\divv A\nabla u=0$ shows that
\begin{equation*}
  \abs{u(z)-u(p_{Q'})}\lesssim\Big(\frac{\abs{z-p_{Q'}}}{\ell(Q')}\Big)^{\alpha}
  \ave{\abs{u}^2}_{\widetilde W_{Q'}}^{1/2}
  \lesssim \eta^{\alpha}\inf_{Q'}Nu\ll \epsilon M_{\mD^\delta}(Nu)(Q'),
\end{equation*}
where $\widetilde{W}_{Q'}$ is a slight expansion of $W_{Q'}$. Thus, we have
\begin{equation*}
  \abs{u(z)-u(p_Q)}\gtrsim\epsilon M_{\mD^\delta}(Nu)(Q')\qquad\text{for all } z\in Q'.\qedhere
\end{equation*}
\end{proof}


The main estimate here is that both the stopping family $\mathcal S$, and the collection of large oscillation cubes $\mathcal R$ introduced in \eqref{eq:defnfamR}, satisfy the Carleson condition.


\begin{lemma}\label{lem:CarlesonS}
For the stopping cubes  $\mathcal{S}$, we have the Carleson measure estimate
\begin{equation*}
  \sum_{S\in\mathcal{S}, S\subset Q_0}\abs{S}\lesssim\abs{Q_0}
\end{equation*}
for all dyadic cubes $Q_0$.
\end{lemma}

\begin{lemma}\label{lem:CarlesonR}
For the large oscillation cubes $\mathcal{R}$, we have the Carleson measure estimate
\begin{equation*}
  \sum_{R\in\mathcal{R}, R\subset Q_0}\abs{R}\lesssim\abs{Q_0}
\end{equation*}
for all dyadic cubes $Q_0$.
\end{lemma}

\begin{proof}[Proof of Lemma~\ref{lem:CarlesonS}]

(A) First we make a preliminary simplification of the estimate based on Lemma~\ref{lem:sortOfVitali}.
By considering the maximal $\mathcal S$-cubes contained in $Q_0$, we may assume without loss of generality that $Q_0\in\mathcal S$. We then write
\begin{equation*}
  \sum_{S\in\mathcal{S}, S\subset Q_0}\abs{S}
  =\sum_{\substack{S\in\mathcal{S}, S\subset Q_0 \\ \pi_{\mathcal P}S=\pi_{\mathcal P}Q_0}}\abs{S}
  +\sum_{P\in\mathcal{P}, P\subsetneqq Q_0} \sum_{\substack{S\in\mathcal{S} \\ \pi_{\mathcal P}S=P}}\abs{S},
\end{equation*}
and we claim that it suffices to prove the required bound $\abs{Q_0}$ for the first term. Namely, if this is done, we simply apply this result, with $Q_0=\pi_{\mathcal P}Q_0=P$, to the inner sum in the second term, which shows that this inner sum is bounded by $\abs{P}$. Then the Carleson property of the collection $\mathcal P$ completes the estimate.

So we concentrate on the first term, and abbreviate $\pi_{\mathcal P}Q_0=:P$ for convenience. We also drop the summation condition ``$S\in\mathcal S$'', with the implicit understanding that this is always in force.

With Lemma~\ref{lem:sortOfVitali} applied to $Q=S$ and $\mathcal{Q}=\operatorname{ch}_{\mathcal S}(S)$ for each relevant $S$, we obtain by indexing the cubes by their parents instead that (Note that each $S$, except for the maximal ones, is a child of another $S$, and the sum over the maximal ones is bounded by $\abs{Q_0}$, by disjointness.)
\begin{equation*}
\begin{split}
  \sum_{S\subset Q_0, \pi_{\mathcal P}S=P}\abs{S}
  &\leq\abs{Q_0}+\sum_{S\subset Q_0, \pi_{\mathcal P}S=P}\sum_{S'\in\operatorname{ch}_{\mathcal S}(S)}\abs{S'} \\
  &\leq\abs{Q_0}+\sum_{S\subset Q_0, \pi_{\mathcal P}S=P}\Big(\tau\abs{S}+C\sum_{S'\in\operatorname{ch}^*_{\mathcal S}(S)}\abs{S'}\Big),
\end{split}
\end{equation*}
where
\begin{equation*}
  \operatorname{ch}^*_{\mathcal S}(S)
  :=\{S'\in\operatorname{ch}_{\mathcal S}(S): S'\text{ centred in $S$ and uncovered by }\operatorname{ch}_{\mathcal S}(S)\}.
\end{equation*}
The second term can then be absorbed to the left, since $\tau<1$.

We further observe the following. If $S'\in\mathcal{P}$ for some $S'$ appearing in the inner sum on the right, this together with $S'\in\operatorname{ch}_{\mathcal S}(S)$ and $\pi_{\mathcal P}S=P$ implies that $S'\in\operatorname{ch}_{\mathcal P}P$. But these cubes are pairwise disjoint. Since $S'\subset S\subset Q_0$, they are also contained in $Q_0$, hence their total volume adds up to at most $\abs{Q_0}$, which may be absorbed into the first term on the right. So altogether we find that
\begin{equation*}
  \sum_{S\subset Q_0, \pi_{\mathcal P}S=P}\abs{S}
  \lesssim\abs{Q_0}+\sum_{S\subset Q_0, \pi_{\mathcal P}S=P}\sum_{S'\in\operatorname{ch}^*_{\mathcal S}(S)\setminus\mathcal P}\abs{S'},
\end{equation*}
and it remains to bound the last double sum by $\abs{Q_0}$.

(B) We now aim to use the local $N\lesssim S$ estimate \eqref{est:N<S}.
We first treat one of the inner sums over $S'\in\operatorname{ch}^*_{\mathcal S}(S)\setminus\mathcal P$ for a fixed $S$. The significance of the restriction $S'\notin\mathcal{P}$ comes from the fact that we then know that $S'$ was chosen as a stopping cube by the criterion \eqref{eq:Sstop}. By Lemma~\ref{lem:SstopBootstrap}, this gives
\begin{equation*}
  M_{\mD^\delta}(Nu)(P)^2\abs{S'}
  \leq M_{\mD^\delta}(Nu)(S')^2\abs{S'}
  \lesssim M_{\mD^\delta}(Nu)(S')^2\abs{\widetilde S'}
  \lesssim \int_{\widetilde S'}\abs{u-u(p_{S})}^2 dx,
\end{equation*}
where we allow the dependence on $\epsilon$ in the implicit constants.

We then consider the Lipschitz function
\begin{equation*}
  \psi_S^1(x):=\sup_{S'\in\operatorname{ch}_{\mathcal S}(S)}\psi_{S'}(x),\qquad\psi_{S'}(x):=\max(\ell(S')-\delta^{-1}\dist(x,S'),0).
\end{equation*}
This is closely related to the notion of coveredness, and illustrated in Figure~\ref{fig:covered}.

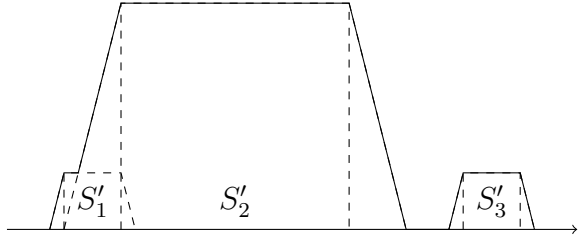
\begin{figure}
\begin{tikzpicture}[scale=3]
  \draw [->] (-0.5,0) -- (2,0);
  \draw [dashed] (0,0) -- (0,1) -- (1,1) -- (1,0);
  \draw [dashed] (-0.25,0) -- (0,1);
  \draw [dashed] (1,1) -- (1.25,0);
  \draw [dashed] (-0.25,0) -- (-0.25,0.25) -- (0,0.25) -- (0.0625,0);
  \draw [dashed] (-0.3125,0) -- (-0.25,0.25);
  \draw [dashed] (1.5,0) -- (1.5,0.25) -- (1.75,0.25) -- (1.75,0);
  \draw [dashed] (1.4375,0) -- (1.5,0.25);
  \draw [dashed] (1.75,0.25) -- (1.8125,0);
  \draw (-0.3125,0) -- (-0.25,0.25) -- (-0.1875,0.25) -- (0,1) -- (1,1) -- (1.25,0) -- (1.4375,0) -- (1.5,0.25) -- (1.75,0.25) -- (1.8125,0);
  \node at (0.5,0.125) {$S_2'$};
  \node at (-0.125,0.125) {$S_1'$};
  \node at (1.625,0.125) {$S_3'$};
\end{tikzpicture}
\caption{\label{fig:covered}A possible configuration of cubes $S_i'\in\operatorname{ch}_{\mathcal{S}}(S)$: $S_2'$ and $S_3'$ are uncovered, but $S_1'$ is covered by $S_2'$. Dashed lines show the Carleson boxes $\widehat S_i'$ and the graphs of $\psi_{S_i'}$, except where overwritten by the continuous line, which shows the graph of $\psi_S^1$.}
\end{figure}

The function $\psi^1_S$ has two important features.
\begin{itemize}
  \item $\widehat S'\cap\Omega_{\psi_S^1}=\emptyset$, where $\Omega_{\psi_S^1}:=\{(t,x)\in\R^{1+n}_+:\psi_S^1(x)<t\}$, for every $S'\in\operatorname{ch}_{\mathcal S}(S)$.
  \item $\psi_S^1(x)=\ell(S')$ for every $x\in S'$, if $S'\in\operatorname{ch}_{\mathcal S}(S)$ is uncovered.
\end{itemize}
This allows us to write, using Lemma~\ref{lem:SstopBootstrap} in the first step,
\begin{equation}\label{eq:useEstimate1}
\begin{split}
  M_{\mD^\delta}(Nu)(P)^2    \sum_{S'\in\operatorname{ch}^*_{\mathcal S}(S)\setminus\mathcal{P} }\abs{S'}
  & \lesssim     \sum_{S'\in\operatorname{ch}^*_{\mathcal S}(S)\setminus\mathcal{P} }
   \int_{\widetilde S'}\abs{u-u(p_S)}^2 dx \\
   &\leq\int_{(1-\delta)S}\abs{u(\psi_S^1(x),x)-u(p_S)}^2 dx \\
   &\lesssim\iint_{\Omega_{\psi_S^1}\cap\widehat{S}}\abs{\nabla u}^2(t-\psi_S^1(x)) dt dx,
\end{split}
\end{equation}
where in the last step we have used the local $N\lesssim S$ estimate \eqref{est:N<S}.

(C) We now aim to use the local $S\lesssim N$ estimate \eqref{est:S<N}.
In order to sum over all relevant $S$, let us next denote
\begin{equation*}
  \psi_P^2(x):=\inf_{Q:\pi_{\mathcal P}Q=P}[\delta\ell(Q)+\dist(x,Q)].
\end{equation*}
Then $\psi_P^2$ is a Lipschitz function, and
\begin{equation*}
  \Omega_{\psi_P^2}:=\{(t,x):\psi_P^2(x)<t\}=\bigcup_{Q:\pi_{\mathcal P}Q=P}\Gamma_Q\supset\bigcup_{Q:\pi_{\mathcal P}Q=P}W_Q,
\end{equation*}
and hence $\abs{u}\lesssim M_{\mD^\delta}(Nu)(P)$ on this set by Lemma~\ref{lem:uControlByMNu} and the stopping condition \eqref{eq:princstopping}.

Returning to \eqref{eq:useEstimate1}, we observe that
\begin{equation*}
  \Omega_{\psi_S^1}\cap\widehat S\subset \widehat S\setminus_{S'\in\operatorname{ch}_{\mathcal S}(S)}\widehat S'
  =\bigcup_{Q:\pi_{\mathcal S}Q=S}W_Q\subset \widehat{S}\cap \Omega_{\psi_P^2},
\end{equation*}
for all $S$ such that $\pi_{\mathcal P}S=P$,
where the first inclusion shows that these sets are pairwise disjoint in $S$. Since both ends of the inclusion involve graph-domains, we also see that $\psi^1_S(x)\ge \psi^2_P(x)$. This allows to estimate and sum over $S$ in \eqref{eq:useEstimate1} as follows.
\begin{equation}\label{eq:useEstimate2}
\begin{split}
   \sum_{S\subset Q_0, \pi_{\mathcal P}S=P} &\iint_{\Omega_{\psi_S^1}\cap\widehat{S}}\abs{\nabla u}^2(t-\psi_S^1(x)) dt dx \\
   &\leq \iint_{\Omega_{\psi_P^2}\cap\widehat{Q}_0}\abs{\nabla u}^2(t-\psi_P^2(x)) dt dx \\
   &\lesssim \abs{Q_0}\cdot\|u\|_{L^\infty(\Omega_{\psi_P^2})}^2
     \lesssim \abs{Q_0}\cdot M_{\mD^\delta}(Nu)(P)^2,
\end{split}
\end{equation}
where the penultimate step is by the local $S\lesssim N$ estimate \eqref{est:S<N}.

A combination of \eqref{eq:useEstimate1} and \eqref{eq:useEstimate2} shows that the factor $M_{\mD^\delta}(Nu)(P)^2$ cancels from both sides, and we are left with
\begin{equation*}
  \sum_{S\subset Q_0, \pi_{\mathcal P}S=P}\sum_{S'\in\operatorname{ch}_{\mathcal S}^*(S)\setminus\mathcal P}\abs{S'}
  \lesssim\abs{Q_0},
\end{equation*}
which was left to prove, to complete the proof of the Lemma.
\end{proof}

\begin{proof}[Proof of Lemma~\ref{lem:CarlesonR}]
By arguing as in the beginning of the proof of Lemma~\ref{lem:CarlesonS}, we find that it is enough to prove that
\begin{equation*}
  \sum_{\substack{R\in\mathcal{R}, R\subset Q_0 \\ \pi_{\mathcal P}R=P }}\abs{R}\lesssim\abs{Q_0},
\end{equation*}
where $P:=\pi_{\mathcal P}Q_0$.

By interior regularity of solutions to $u$ to $\divv A\nabla u=0$, it follows that for all $z,w\in W_R$,
\begin{equation*}
\begin{split}
  \abs{u(z)-u(w)}
  &\lesssim\Big(\frac{\abs{z-w}}{\ell(R)}\Big)^{\alpha}\ell(R)\Big(\frac{1}{\abs{W_R}}\iint_{\widetilde W_R}\abs{\nabla u}^2 dtdx\Big)^{1/2} \\
  &\lesssim\Big(\frac{1}{\abs{R}}\iint_{\widetilde W_R}\abs{\nabla u}^2 t \,dtdx\Big)^{1/2},
\end{split}
\end{equation*}
where $\widetilde W_R$ is a slight expansion of the Whitney rectangle $W_R$. Hence for $R\in\mathcal R$, we have
\begin{equation}\label{eq:MvsGrad}
  M_{\mD^\delta}(Nu)(R)^2\abs{R}\lesssim\big(\osc_{W_R}u\big)^2\abs{R}
  \lesssim\iint_{\widetilde W_R}\abs{\nabla u}^2 t \,dtdx.
\end{equation}

Let $W_R^*=[\delta'\ell(R),\kappa'\ell(R))\times R^*$ be a slightly bigger expansion and $R^*$ its projection onto $\R^n$, where $\delta'\in(0,\delta)$. Let
\begin{equation*}
  \psi_R^*(x):=\delta'\ell(R^*)+\dist(x,R^*),\qquad
  \Gamma_R^*:=\{(t,x):t>\psi_R^*(x)\}.
\end{equation*}
Then $\widetilde{W}_R\subset\Gamma_R^*$ and
\begin{equation}\label{eq:tVstMinus}
   t\lesssim t-\psi_R^*(x)\quad\text{for all }(t,x)\in\widetilde{W}_R. 
\end{equation}
Let further
\begin{equation*}
  \psi_P^{**}(x):=\inf_{Q:\pi_{\mathcal P}Q=P}\psi_R^*(x),
\end{equation*}
so that
\begin{equation*}
  \bigcup_{Q:\pi_{\mathcal P}Q=P}W_Q\subset\bigcup_{Q:\pi_{\mathcal P}Q=P}\Gamma_Q^*=\{(t,x):t>\psi_P^{**}(x)\}=:\Omega_{\psi_P^{**}}.
\end{equation*}
It follows that
\begin{equation}\label{eq:CarlesonRmainStep}
\begin{split}
  M_{\mD^\delta}(Nu)(P)^2 &\sum_{\substack{R\in\mathcal{R}, R\subset Q_0\\ \pi_{\mathcal P}R=P}}\abs{R} \\
  &\leq\sum_{\substack{R\in\mathcal{R}, R\subset Q_0\\ \pi_{\mathcal P}R=P}}M_{\mD^\delta}(Nu)(R)^2\abs{R} \\
  &\lesssim\sum_{\substack{R\in\mathcal{R}, R\subset Q_0\\ \pi_{\mathcal P}R=P}}\iint_{\widetilde W_R}\abs{\nabla u}^2(t-\psi_P^{**})dtdx
     \qquad\text{by \eqref{eq:MvsGrad} and \eqref{eq:tVstMinus}} \\
  &\overset{(*)}{\lesssim}\iint_{\Omega_{\psi_P^{**}}\cap\widehat{\widetilde{Q}_0}}\abs{\nabla u}^2(t-\psi_P^{**})dtdx \\
  &\lesssim\abs{Q_0}\cdot\|u\|_{L^\infty(\Omega_{\psi_P^{**}})}^2\qquad\text{by the local $S<N$ bound \eqref{est:S<N}} \\
   &\overset{(**)}{\lesssim}\abs{Q_0}\cdot M_{\mD^\delta}(Nu)(P)^2.
\end{split}
\end{equation}
In $(*)$, we used the bounded overlap of the regions $\widetilde{W}_R$, which is an easy consequence of the geometry of the Whitney regions, and their containment in $\widehat{\widetilde{Q}_0}$, a slight expansion of the Carleson box $\widehat{Q_0}$. In the last step $(**)$, we used the fact that $\Gamma_R^*\subset\Gamma_\gamma(x):=\{(t,y):\abs{y-x}<\gamma t\}$ for all $x\in R$, provided that $\gamma$ is large enough, and therefore
\begin{equation*}
  \sup_{\Gamma_R^*}\abs{u}\leq \inf_{x\in R}Nu(x)\leq M_{\mD^\delta}(Nu)(R)\lesssim M_{\mD^\delta}(Nu)(P)
\end{equation*}
whenever $\pi_{\mathcal P}R=P$, provided that the aperture defining $Nu$ is large enough.

Observing that $M_{\mD^\delta}(Nu)(P)^2$ cancels from both sides of \eqref{eq:CarlesonRmainStep}, we have established the required bound.
\end{proof}

\subsection{The $\epsilon$-approximating functions}  \label{subsec:approximant}

As a first approximation, consider the piecewise constant function
\begin{equation*}
  \varphi_1:=u(p_S)\quad\text{on}\quad\Omega_{\mathcal{S}}(S):=\bigcup_{Q:\pi_{\mathcal S}Q=S}W_Q.
\end{equation*}

\begin{lemma}  \label{lem:phi1}
We have the estimate
\begin{equation*}
  \iint\abs{\nabla(1_{\widehat Q_0}\varphi_1)}dtdx\lesssim \int_{Q_0}Nu\, dx
\end{equation*}
for all dyadic cubes $Q_0$.
\end{lemma}


\begin{proof}
Let us abbreviate $S_0:=\pi_{\mathcal S}Q_0$. Then we have
\begin{equation}\label{eq:piecewiseJumpEst}
\begin{split}
  1_{\widehat Q_0}\varphi_1
  &=\sum_{S\in \mathcal{S}, S\subset Q_0}u(p_S)1_{\Omega_{\mathcal S}(S)}+u(p_{S_0}) 1_{\Omega_{\mathcal S}(S_0)\cap \widehat Q_0}, \\
  \abs{\nabla(1_{\widehat Q_0}\varphi_1)}
  &\leq \sum_{S\in \mathcal{S}, S\subset Q_0}\abs{u(p_S)}\cdot\abs{\nabla 1_{\Omega_{\mathcal S}(S)}}+\abs{u(p_{S_0})}\cdot\abs{\nabla 1_{\Omega_{\mathcal S}(S_0)\cap \widehat Q_0}} \\
  &= \sum_{S\in \mathcal{S}, S\subset Q_0}\abs{u(p_S)}\cdot H^n\lfloor \partial\Omega_{\mathcal S}(S) +\abs{u(p_{S_0})}\cdot H^n\lfloor \Omega_{\mathcal S}(S_0)\cap \widehat Q_0,
\end{split}
\end{equation}
where $H^n$ is the $n$-dimensional Hausdorff measure, or more simply, the $n$-dimensional Lebesgue measure on the hyperplanes to which it is restricted.

We can then compute
\begin{equation*}
\begin{split}
  \iint\abs{\nabla(1_{\widehat Q_0}\varphi_1)}dtdx
  &\lesssim \sum_{S\in \mathcal{S}, S\subset Q_0}\abs{u(p_S)}\cdot \abs{S} +\abs{u(p_{S_0})}\cdot \abs{Q_0} \\
  &\leq\sum_{S\in \mathcal{S}, S\subset Q_0}\inf_S Nu \cdot \abs{S} +\inf_{Q_0}Nu\cdot \abs{Q_0} \\
  &\lesssim\int_{Q_0} Nu\, dx,
\end{split}
\end{equation*}
where the first estimate is based on simple geometric observations concerning the shape of the sets $\Omega_{\mathcal S}(S)$ and $\Omega_{\mathcal S}(S_0)\cap \widehat Q_0$, and the last one on the Carleson inequality and Carleson property of the collection $\mathcal S$ for the first term from Lemma~\ref{lem:CarlesonS}, and a trivial estimate for the second.
\end{proof}

\begin{rem}
It is perhaps interesting to remark that, in bounding the gradient as in \eqref{eq:piecewiseJumpEst}, we make an apparently crude estimate of the jumps of $\varphi_1$ in the interior of $\widehat Q_0$, in that we dominate a jump $\abs{u(p_{S'})-u(p_S)}$ simply by $\abs{u(p_{S'})}+\abs{u(p_S)}$. However, a comparison with the stopping criterion \eqref{eq:Sstop} shows that this is not so crude after all: it is easy to check that $\abs{u(p_{S'})}+\abs{u(p_S)}\lesssim M_{\mathcal{D}^\delta}(Nu)(S')$ for all $S'\subset S$, and the very stopping criterion \eqref{eq:Sstop} says that, for consecutive stopping cubes $S'\subsetneq S$, the difference is already essentially as big as this maximal quantity. This means that, if we ignore the dependence on $\epsilon$ as we do, there is no essential loss in making this apparently crude estimate. Note, however, that we argued somewhat differently in the context of dyadic martingales, where we did trace a good dependence on $\epsilon$.
\end{rem}

The function $\varphi_1$ provides a good $\epsilon$-approximation of $u$ in all those $W_Q$ where
\begin{equation*}
  \osc_{W_Q}u<\epsilon M_{\mD^\delta}(Nu)(Q),
\end{equation*}
that is, whenever $Q\notin\mathcal R$. Likewise, it is clear that $\varphi_1$ fails to be a good approximation in any $W_R$ with $R\in\mathcal R$. Our final $\epsilon$-approximation will be $\varphi_1+\varphi_2$, where
\begin{equation*}
  \varphi_2|_{W_Q}:=\begin{cases} (u-\varphi_1)|_{W_Q}=u|_{W_Q}-u(p_{\pi_{\mathcal S}Q}), &\text{if } Q\in\mathcal R \\
    0, & \text{else}.\end{cases}
\end{equation*}
It remains to show that $\varphi_2$ satisfies the needed Carleson measure estimate.

\begin{lemma}  \label{lem:phi2}
We have the estimate
\begin{equation*}
   \iint\abs{\nabla(1_{\widehat Q_0}\varphi_2)} dtdx
  \lesssim\int_{Q_0}Nu\, dx
\end{equation*}
for all dyadic cubes $Q_0$.
\end{lemma}

\begin{proof}
We have
\begin{equation*}
\begin{split}
  1_{\widehat Q_0}\varphi_2 &=\sum_{R\in\mathcal R, R\subset Q_0}(u-u(p_{\pi_{\mathcal S}R}))\cdot 1_{W_R}, \\
  \nabla(1_{Q_0}\varphi_2) &=\sum_{R\in\mathcal R, R\subset Q_0}[\nabla u \cdot 1_{W_R}+(u-u(p_{\pi_{\mathcal S}R}))\cdot \nabla 1_{W_R}] ,
\end{split}
\end{equation*}
and hence
\begin{equation}\label{eq:phi2MainStep}
  \iint\abs{\nabla(1_{Q_0}\varphi_2)}dtdx \lesssim
  \sum_{R\in\mathcal R, R\subset Q_0}\iint_{W_R}\abs{\nabla u}dtdx
  +\sum_{R\in\mathcal R, R\subset Q_0}\inf_R Nu\cdot \abs{R},
\end{equation}
using again that
\begin{equation*}
  \abs{\nabla 1_{W_R}}=H^n\lfloor\partial W_R,\qquad H^n(\partial W_R)\lesssim\abs{R}.
\end{equation*}
By Caccioppoli's inequality, we can estimate the first term in \eqref{eq:phi2MainStep} by
\begin{equation*}
\begin{split}
  \iint_{W_R}\abs{\nabla u}dtdx
  &\leq\Big(\iint_{W_R}\abs{\nabla u}^2dtdx\Big)^{1/2}\abs{W_R}^{1/2}\\
  &\lesssim\frac{1}{\ell(R)}\Big(\iint_{\widetilde W_R}\abs{u}^2dtdx\Big)^{1/2}\abs{W_R}^{1/2} \\
  &\lesssim \frac{1}{\ell(R)}\Big(\iint_{\widetilde W_R}\inf_R (Nu)^2dtdx\Big)^{1/2}\abs{W_R}^{1/2} \\
   &\lesssim \frac{1}{\ell(R)}\inf_R (Nu)\abs{W_R}=\inf_R(Nu)\abs{R},
\end{split}
\end{equation*}
which coincides with the second term in \eqref{eq:phi2MainStep}. So altogether
\begin{equation*}
  \iint\abs{\nabla(1_{Q_0}\varphi_2)}dtdx
  \lesssim\sum_{R\in\mathcal{R}, R\subset Q_0}\inf_R Nu\cdot\abs{R}
  \lesssim\int_{Q_0}Nu\,dx,
\end{equation*}
by Carleson's inequality and the Carleson property of $\mathcal R$ from Lemma~\ref{lem:CarlesonR}
in the last step.
\end{proof}

\bibliographystyle{acm}

\end{document}